\documentclass[12pt]{amsart}

\usepackage{amssymb,amsmath,graphicx, amsthm, MnSymbol, array}
\usepackage[all]{xy}
\newtheorem{theorem}{Theorem}[section]
\newtheorem{lemma}[theorem]{Lemma}
\newtheorem{proposition}[theorem]{Proposition}
\newtheorem{corollary}[theorem]{Corollary}

\newtheorem{claim}[theorem]{Claim}
\theoremstyle{definition}
\newtheorem{definition}[theorem]{Definition}

\newtheorem{example}[theorem]{Example}

\newtheorem{remark}[theorem]{Remark}

\newenvironment{proofclaim}{\paragraph{\emph{Proof of the Claim}.}}{\hfill$\qed$\\}

\newcommand{\thd}{{\twoheaddownarrow}}
\newcommand{\thu}{{\twoheaduparrow}}
\def\int{{\sf int}}
\newcommand\cl{{\sf cl}}

\newcommand{\dv}[1]{\mathfrak{#1}}
\newcommand\RO{\mathcal{RO}}

\newcommand\dev{{\sf DeV}}
\newcommand\deve{{\sf DeVe}}
\newcommand\ldeve{{\sf LDeVe}}
\newcommand\ndeve{{\sf NDeVe}}
\newcommand\Mdeve{{\sf MDeVe}}
\newcommand\Cdeve{{\sf CDeVe}}
\newcommand\C{{\sf Comp}}
\newcommand\creg{{\sf CReg}}
\newcommand\norm{{\sf Norm}}
\newcommand\KHaus{{\sf KHaus}}
\newcommand\LKHaus{{\sf LKHaus}}

\newcommand\Dim{{\sf Dim}}

\newcolumntype{x}[1]{>{\raggedright\arraybackslash}p{#1}}

\begin{document}

\title{De Vries duality for normal spaces and locally compact Hausdorff spaces}
\author{G.~Bezhanishvili, P.~J.~Morandi, B.~Olberding}
\date{}


\begin{abstract}
By de Vries duality, the category of compact Hausdorff spaces is dually equivalent to the category of de Vries algebras. In our recent article, we have extended de Vries duality to completely regular spaces by generalizing de Vries algebras to de Vries extensions. To illustrate the utility of this point of view, we show how to use this new duality to obtain algebraic counterparts of normal and locally compact Hausdorff spaces in the form of de Vries extensions that are subject to additional axioms which encode the desired topological property. This, in particular, yields a different perspective on de Vries duality. As a further application, we show that a duality for locally compact Hausdorff spaces due to Dimov can be obtained from our approach.
\end{abstract}

\maketitle

\section{Introduction}

It is a well-known theorem of Smirnov that compactifications of a completely regular space $X$ can be described ``internally" by means of proximities
on $X$ compatible with the topology on $X$ (see, e.g., \cite[Sec.~7]{NW70}), where a proximity is a binary relation on the powerset of $X$ satisfying
certain natural axioms, including a point-separation axiom (see, e.g., \cite[Sec.~3]{NW70}). De Vries takes this further in \cite{deV62} by axiomatizing
the proximities on the complete Boolean algebra of regular open subsets of $X$ that correspond to compactifications of $X$. In de Vries' work, the
point-separation axiom
is replaced by the point-free axiom asserting that the proximity relation is approximating. This point-free axiom decouples the proximity from the
underlying space and yields what is known today as a de Vries algebra: a complete Boolean algebra with a binary relation satisfying all of the axioms
of a proximity except the point-separation axiom, which is replaced by de Vries' point-free axiom. De Vries showed that this axiomatization can be
used to give an algebraic description of the category ${\KHaus}$ of compact Hausdorff spaces. More formally, the category ${\KHaus}$ is dually
equivalent to the category $\dev$ of de Vries algebras.

In \cite{BMO18a} we extended de Vries duality to completely regular spaces by generalizing the notion of a de Vries algebra to that of a de Vries
extension.
While de Vries duality alone is not enough to deal with completely regular spaces and de Vries extensions, we show in \cite{BMO18a} that
de Vries duality together with Tarski duality for complete and atomic Boolean algebras provides an appropriate framework for dealing with completely
regular spaces. In fact, the methods in \cite{BMO18a} yield a dual equivalence between the category of de Vries extensions and the category
of compactifications of completely regular spaces that extends both de Vries duality and Tarski duality. The de Vries extensions corresponding
to Stone-\v{C}ech compactifications are axiomatized in \cite{BMO18a} as ``maximal'' de Vries extensions. This in turn yields a dual equivalence
between the category of completely regular spaces and the category of maximal de Vries extensions, thereby providing an algebraic counterpart
to completely regular spaces.

It is noted in \cite{BMO18a} that discrete spaces can be viewed algebraically  as trivial de Vries extensions. The interpretation of more
interesting classes of completely regular spaces is not as straightforward. In this article, we continue our work begun in \cite{BMO18a}
by giving algebraic interpretations of normal spaces and locally compact Hausdorff spaces
within our framework of de Vries extensions. On a technical level, this involves, for a compactification $e:X \rightarrow Y$ of a completely
regular space $X$, a close analysis of the corresponding de Vries extension $e^{-1}:(\RO(Y), \prec) \rightarrow (\wp(X),\subseteq)$, where $\RO(Y)$
is the complete Boolean algebra of regular open sets of $Y$ and $\wp(X)$ is the powerset of $X$. Some of the main results of the current paper
involve determining which algebraic properties of the map $e^{-1}$ reflect the normality
and local compactness of $X$. With these characterizations, we obtain dual equivalences between the categories of such spaces and the appropriate
full subcategories of the category of de Vries extensions.

The article is organized as follows. 
In Section~2 we recall all the necessary background about de Vries algebras, de Vries extensions, and maximal
de Vries extensions. In Section~3 we introduce normal de Vries extensions and show that every normal de Vries extension is maximal. This allows us
to view normal de Vries extensions as a full subcategory $\ndeve$ of the category $\Mdeve$ of maximal de Vries extensions. We prove that $\ndeve$
is dually equivalent to the category $\norm$ of normal spaces. In Section~4 we introduce locally compact de Vries extensions. While not every locally
compact de Vries extension is maximal, we prove that $\LKHaus$ is dually equivalent to the full subcategory $\ldeve$ of $\Mdeve$ consisting of locally
compact de Vries extensions.

In Section~5 we introduce compact de Vries extensions. Since every compact de Vries extension is maximal, we view compact de Vries extensions as a
full subcategory $\Cdeve$ of $\Mdeve$. We prove that $\Cdeve$ is equivalent to $\dev$, and hence dually equivalent to $\KHaus$. This gives another
perspective on de Vries duality. In Section~6 we introduce minimal de Vries extensions and show that non-compact minimal de Vries extensions 
correspond to one-point compactifications of non-compact locally compact Hausdorff spaces.

While we do not make use of it directly, we are motivated in our treatment of local compactness by Leader's generalization in \cite{Lea67} of
a proximity relation to that of a local proximity relation, a generalization that yields a description of the local compactifications of a
completely regular space by means of local proximity relations compatible with the topology. Recently, Dimov \cite{Dim10} has recast Leader's
work in a setting similar to de Vries algebras, and obtained a duality theorem for the category $\LKHaus$ of locally compact Hausdorff spaces
that generalizes de Vries duality. In Section~7 we show that Dimov's duality
for $\LKHaus$ can be derived as a consequence of our duality for $\LKHaus$.

\section{De Vries extensions}

In this preliminary section we recall de Vries algebras, de Vries extensions, and maximal de Vries extensions. By de Vries duality \cite{deV62}, de Vries algebras provide an algebraic counterpart of compact Hausdorff spaces. By the duality developed in \cite{BMO18a}, de Vries extensions provide an algebraic counterpart of compactifications of completely regular spaces. Under this duality, maximal de Vries extensions correspond to Stone-\v{C}ech compactifications, thus yielding a duality for completely regular spaces that generalizes de Vries duality for compact Hausdorff spaces.

\subsection{De Vries algebras and compact Hausdorff spaces}

We start by recalling de Vries algebras and de Vries morphisms.

\begin{definition}\label{def: DeV}
\begin{enumerate}
\item[]
\item A \emph{de Vries algebra} is a pair $\mathfrak A=(A,\prec)$, where $A$ is a complete Boolean algebra and $\prec$ is a
binary relation on $A$ satisfying the following axioms:
\begin{enumerate}
\item[(DV1)] $1\prec 1$.
\item[(DV2)] $a\prec b$ implies $a\le b$.
\item[(DV3)] $a\le b\prec c\le d$ implies $a\prec d$.
\item[(DV4)] $a\prec b,c$ implies $a\prec b\wedge c$.
\item[(DV5)] $a\prec b$ implies $\neg b\prec \neg a$.
\item[(DV6)] $a\prec b$ implies there is $c$ such that $a\prec c\prec b$.
\item[(DV7)] $b = \bigvee\{a\in A \mid a\prec b\}$.
\end{enumerate}
\item A \emph{de Vries morphism} is a map $\rho: \dv{A} \to \dv{B}$ between de Vries algebras satisfying the following axioms:
\begin{enumerate}
\item[(M1)] $\rho(0)=0$.
\item[(M2)] $\rho(a\wedge b)=\rho(a)\wedge\rho(b)$.
\item[(M3)] $a\prec b$ implies $\neg\rho(\neg a)\prec\rho(b)$.
\item[(M4)] $\rho(b) = \bigvee \{\rho(a)\mid a\prec b\}$.
\end{enumerate}
\end{enumerate}
\end{definition}

De Vries algebras and de Vries morphisms form a category $\dev$ where the composition of two de Vries morphisms $\rho_1:\dv{A}_1\to \dv{A}_2$ and $\rho_2: \dv{A}_2 \to \dv{A}_3$ is defined by
\[
(\rho_2 \star \rho_1)(b)=\bigvee\{\rho_2\rho_1(a)\mid  a\prec b\}.
\]

De Vries algebras arise naturally from compact Hausdorff spaces. If $X$ is a compact Hausdorff space, then the pair $X^*=(\mathcal{RO}(X),\prec)$ is a de Vries algebra, where $\mathcal{RO}(X)$ is the complete Boolean algebra of regular open subsets of $X$ and $\prec$ is the canonical proximity relation on $\mathcal{RO}(X)$ given by
\[
U\prec V \mbox{ iff } {\cl}(U)\subseteq V.
\]
Similarly, de Vries morphisms arise naturally from continuous maps between compact Hausdorff spaces. If $f:X\to Y$ is such a map, then $f^*:\mathcal{RO}(Y)\to\mathcal{RO}(X)$ is a de Vries morphism, where $f^*$ is given by
\[
f^*(U)={\int}\left({\cl}\left(f^{-1}(U)\right)\right).
\]
This defines a contravariant functor $(-)^*:\KHaus\to\dev$. To define a contravariant functor $(-)_*:\dev\to\KHaus$, let $\dv{A}$ be a de Vries algebra. For $S\subseteq \dv{A}$, let
$$
\thu S = \{a\in \dv{A} \mid  b\prec a \mbox{ for some } b\in S\}
$$
and
$$
\thd S = \{a\in \dv{A} \mid  a\prec b \mbox{ for some } b\in S\}.
$$
A filter $F$ of $\dv{A}$ is \emph{round} if $\thu F=F$. Similarly, an ideal $I$ of $\dv{A}$ is \emph{round} if $\thd I=I$.

An \emph{end} of $\dv{A}$ is a maximal proper round filter. Let $Y_\dv{A}$ be the set of ends of $\dv{A}$.
For $a\in \dv{A}$, set
$$
\zeta_\dv{A}(a)=\{x\in Y_\dv{A} \mid a\in x\},
$$
and define a topology on $Y_\dv{A}$ by letting
$$
\zeta_\dv{A}[\dv{A}]=\{\zeta_\dv{A}(a) \mid a\in \dv{A}\}
$$
be a basis for the topology. Then $Y_\dv{A}$ is compact Hausdorff, and we set $\dv{A}_*=Y_\dv{A}$. For a de Vries morphism $\rho : \dv{A} \to \dv{A}'$, define $\rho_* : Y_\dv{A'} \to Y_\dv{A}$ by $\rho_*(y) = \thu \rho^{-1}(y)$. Then $\rho_*$ is a well-defined continuous map, yielding a contravariant functor $(-)_* : \dev \to \KHaus$.

We have that $\zeta_{\dv{A}} : \dv{A} \to (\dv{A}_*)^* $ is a de Vries isomorphism, and $\xi_X : X \to (X^*)_*$ is a homeomorphism, where $\xi(x) = \{ U \in X^* \mid x \in U\}$. Therefore, $\zeta : 1_\dev \to (-)^* \circ (-)_*$ and $\xi : 1_\KHaus \to (-)_* \circ (-)^*$ are natural isomorphisms. Thus, we arrive at de Vries duality.

\begin{theorem}
{\em (de Vries \cite{deV62})} $\KHaus$ is dually equivalent to $\dev$.
\end{theorem}

\subsection{De Vries extensions and compactifications}

We next generalize de Vries algebras to de Vries extensions \cite{BMO18a}. For this we will utilize Tarski duality between the category $\sf CABA$ of complete and atomic Boolean algebras with complete Boolean homomorphisms and the category $\sf Set$ of sets and functions. If $X$ is a set, then $\wp(X)$ is a complete and atomic Boolean algebra, and if $f : X \to Y$ is a function, then $f^{-1} : \wp(Y) \to \wp(X)$ is a complete Boolean homomorphism. This yields a contravariant functor $ {\sf Set} \to {\sf CABA}$. Going backwards, for a complete and atomic Boolean algebra $\dv{B}$, let $X_\dv{B}$ be the set of atoms of $\dv{B}$, and for a complete Boolean homomorphism $\sigma : \dv{B}_1 \to \dv{B}_2$, let $\sigma_+ : X_{\dv{B}_2} \to X_{\dv{B}_1}$ be given by $\sigma_+(x) = \bigwedge \{ b \in \dv{B}_1 \mid x \le \sigma(b) \}$. It is well known that $\sigma_+$ is a well-defined function, yielding a contravariant functor ${\sf CABA} \to {\sf Set}$. For each set $X$, we have a natural isomorphism $\eta_X : X \to X_{\wp(X)}$,
given by $\eta_X(x) = \{x\}$ for each $x \in X$; and for each $\dv{B} \in \sf CABA$, we have a natural isomorphism $\vartheta_\dv{B} : \dv{B} \to \wp(X_\dv{B})$, given by $\vartheta_\dv{B}(b) = \{ x \in X_\dv{B} \mid x \le b\}$.

\begin{definition}
\begin{enumerate}
\item[]
\item A de Vries algebra $\mathfrak A=(A,\prec)$ is \emph{extremally disconnected} if $a\prec b$ iff $a\le b$.
\item A de Vries algebra $\mathfrak A=(A,\prec)$ is \emph{atomic} if $A$ is atomic as a Boolean algebra.
\item A \emph{de Vries extension} is a 1-1 de Vries morphism $\alpha: \dv{A}\to \dv{B}$ such that $\dv{A}$ is a de Vries algebra, $\dv{B}$ is an atomic extremally disconnected de Vries algebra, and $\alpha[\dv{A}]$ is join-meet dense\footnote{That is, each element of $\dv{B}$ is a join of meets of elements of $\alpha[\dv{A}]$. This is equivalent to each element of $\dv{B}$ being a meet of joins of elements of $\alpha[\dv{A}]$ (see \cite[Rem.~4.7]{BMO18a}).} in $\dv{B}$.
\end{enumerate}
\end{definition}

A morphism between de Vries extensions $\alpha: \dv{A}\to \dv{B}$ and $\alpha': \dv{A'}\to \dv{B'}$ is a pair $(\rho, \sigma)$, where $\rho:\dv{A}\to\dv{A'}$ is a de Vries morphism, $\sigma:\dv{B}\to\dv{B'}$ is a complete Boolean homomorphism, and $\sigma\circ\alpha = \alpha' \star \rho$.
\[
\xymatrix@C5pc{
\dv{A} \ar[r]^{\alpha} \ar[d]_{\rho} & \dv{B} \ar[d]^{\sigma} \\
\dv{A'} \ar[r]_{\alpha'} & \dv{B'}
}
\]
Since $\sigma$ is a complete Boolean homomorphism and $\dv{B}$, $\dv{B}'$ are extremally disconnected, $\sigma$ is a de Vries morphism and $\sigma \star \alpha = \sigma \circ \alpha$; hence, if $(\rho, \sigma)$ is a morphism in $\deve$, then the diagram above commutes in $\dev$ (see \cite[Rems.~2.6, 4.10]{BMO18a}).

The composition of two morphisms $(\rho_1, \sigma_1)$ and $(\rho_2, \sigma_2)$ is defined as $(\rho_2\star \rho_1, \sigma_2\circ \sigma_1)$.
\[
\xymatrix@C5pc{
\dv{A_1} \ar@/_1.5pc/[dd]_{\rho_2\star\rho_1}\ar[r]^{\alpha_1} \ar[d]^{\rho_1} & \dv{B}_1 \ar[d]_{\sigma_1} \ar@/^1.5pc/[dd]^{\sigma_2\circ\sigma_1} \\
\dv{A}_2 \ar[r]^{\alpha_2} \ar[d]^{\rho_2} & \dv{B}_2 \ar[d]_{\sigma_2} \\
\dv{A}_3 \ar[r]_{\alpha_3} & \dv{B}_3
}
\]
It is straightforward to see that de Vries extensions with morphisms between them form a category, which we denote $\deve$.

De Vries extensions arise naturally from compactifications of completely regular spaces. Let $e : X \to Y$ be a compactification of a completely regular space $X$. Then $(\RO(Y), \prec)$ is a de Vries algebra, the powerset $(\wp(X), \subseteq)$ is an atomic extremally disconnected de Vries algebra, and the pullback map $e^{-1} : \RO(Y) \to \wp(X)$ is a de Vries extension.

Let $\C$ be the category whose objects are compactifications $e:X\to Y$ and whose morphisms are pairs $(f,g)$ of continuous maps such that the following diagram commutes.
\[
\xymatrix@C5pc{
X \ar[r]^{e} \ar[d]_{f} & Y \ar[d]^{g} \\
X' \ar[r]_{e'} & Y'
}
\]
The composition of two morphisms $(f_1,g_1)$ and $(f_2,g_2)$ in $\C$ is $(f_2\circ f_1, g_2\circ g_1)$.
\[
\xymatrix@C5pc{
X_1 \ar[r]^{e_1} \ar@/_1.5pc/[dd]_{f_2\circ f_1} \ar[d]^{f_1} & Y_1 \ar[d]_{g_1} \ar@/^1.5pc/[dd]^{g_2\circ g_1}\\
X_2 \ar[r]^{e_2} \ar[d]^{f_2} & Y_2 \ar[d]_{g_2} \\
X_3 \ar[r]_{e_3} & Y_3
}
\]

For a morphism $(f,g)$ in $\C$, the pair $(g^*,f^{-1})$ is a morphism in $\deve$, where $g^*$ is the de Vries dual of $g$.
\[
\xymatrix@C4pc{
\RO(Y') \ar[d]_{g^*} \ar[r]^{(e')^{-1}} & \wp(X') \ar[d]^{f^{-1}}\\
\RO(Y) \ar[r]_{e^{-1}} & \wp(X)
}
\]
This yields a contravariant functor ${\sf E}:\C\to\deve$. To define a contravariant functor ${\sf C}:\deve\to\C$, let $\alpha:\dv{A}\to\dv{B}$ be a de Vries extension. Let $X_\dv{B}$ be the set of atoms of $\dv{B}$. For $b\in X_\dv{B}$, we have ${\uparrow}b$ is an ultrafilter of $\dv{B}$, and we define $\alpha_*:X_\dv{B}\to Y_\dv{A}$ by
\[
\alpha_*(b)=\thu\alpha^{-1}({\uparrow}b).
\]
We can view $X_\dv{B}$ as a subset of $Y_\dv{B}$ by sending $b$ to ${\uparrow}b$. Then we can think of $\alpha_*$ as the restriction to $X_\dv{B}$ of the de Vries dual $\alpha_* : Y_\dv{B} \to Y_\dv{A}$. By \cite[Lem.~5.4]{BMO18a}, $\alpha_*$ is 1-1. Let $\tau_\alpha$ be the least topology on $X_\dv{B}$ making $\alpha_*$ continuous. By \cite[Thm.~5.7]{BMO18a}, $\alpha_*:X_\dv{B}\to\dv{A}$ is a compactification. For a morphism $(\rho, \sigma)$ in $\deve$
\[
\xymatrix@C5pc{
\dv{A} \ar[r]^{\alpha} \ar[d]_{\rho} & \dv{B} \ar[d]^{\sigma} \\
\dv{A}' \ar[r]_{\alpha'} & \dv{B}'
}
\]
the pair $(\sigma_+, \rho_*)$ is a morphism in $\C$,
\[
\xymatrix@C5pc{
X_{\dv{B'}} \ar[r]^{\alpha'_*} \ar[d]_{\sigma_+} & Y_\dv{A'} \ar[d]^{\rho_*} \\
X_{\dv{B}} \ar[r]_{\alpha_*} & Y_\dv{A}
}
\]
where $\rho_*$ is the de Vries dual of $\rho$ and $\sigma_+$ is the Tarski dual of $\sigma$.
This yields a contravariant functor ${\sf C}:\deve\to\C$, and the functors $\sf E$ and $\sf C$ establish a dual equivalence between $\C$ and $\deve$:

\begin{theorem} \label{thm: duality}
\cite[Thm.~5.9]{BMO18a} $\C$ is dually equivalent to $\deve$.
\end{theorem}

\subsection{Maximal de Vries extensions, Stone-\v{C}ech compactifications, and completely regular spaces}

We next turn to maximal de Vries extensions.

\begin{definition}
\begin{enumerate}
\item[]
\item We call two de Vries extensions $\alpha : \dv{A} \to \dv{B}$ and $\gamma : \dv{C} \to \dv{B}$ \emph{compatible} if
$\alpha[\dv{A}]=\gamma[\dv{C}]$.
\item We say that a de Vries extension $\alpha : \dv{A} \to \dv{B}$ is \emph{maximal} provided for every compatible de Vries extension $\gamma : \dv{C} \to \dv{B}$, there is a de Vries morphism $\delta: \dv{C} \to \dv{A}$ such that $\alpha \star \delta = \gamma$.
\[
\xymatrix{
\dv{A} \ar[rr]^{\alpha} && \dv{B} \\
& \dv{C} \ar[lu]^{\delta} \ar[ru]_{\gamma} &
}
\]
\item Let $\Mdeve$ be the full subcategory of $\deve$ consisting of maximal de Vries extensions.
\end{enumerate}
\end{definition}

\begin{remark} \label{rem: compatible}
Let $\alpha : \dv{A} \to \dv{B}$ and $\gamma : \dv{C} \to \dv{B}$ be compatible de Vries extensions. Since $\alpha$ and $\gamma$ are 1-1 with the
same image, we have a bijection $\delta := \alpha^{-1}\circ \gamma : \dv{C} \to \dv{A}$ with inverse $\gamma^{-1}\circ\alpha : \dv{A} \to \dv{C}$.
Because $\alpha$ and $\gamma$ are both 1-1 and meet preserving, they are both order preserving and order reflecting, so $\delta$ and its inverse
are poset isomorphisms, hence Boolean isomorphisms.
\end{remark}

\begin{theorem} \label{thm: Stone-Cech}
\cite[Thm.~6.4]{BMO18a} If $e : X \to Y$ is a compactification, then the associated de Vries extension $e^{-1}: \RO(Y) \to \wp(X)$ is maximal iff $e:X\to Y$ is isomorphic in $\C$ to the Stone-\v{C}ech compactification $s : X\to\beta X$.
\end{theorem}

\begin{remark} \label{rem: equiv versus iso}
Equivalent compactifications $e : X \to Y$ and $e' : X \to Y'$ are isomorphic in $\C$ but the converse is not true in general \cite[Ex.~3.2]{BMO18a}. However, if $e'$ is the Stone-\v{C}ech compactification of $X$, then $e$ is equivalent to $e'$ iff $e$ is isomorphic to $e'$ in $\C$ \cite[Thm.~3.3]{BMO18a}.
\end{remark}

Let $\creg$ be the category of completely regular spaces and continuous maps. Sending a completely regular space $X$ to its Stone-\v{C}ech compactification $s : X\to\beta X$ yields an equivalence between $\creg$ and the full subcategory of $\C$ consisting of Stone-\v{C}ech compactifications. Since Stone-\v{C}ech compactifications dually correspond to maximal de Vries extensions, we arrive at the following duality theorem, which generalizes de Vries duality to completely regular spaces.

\begin{theorem}\label{thm: comp reg}
\cite[Thm.~6.9]{BMO18a} $\creg$ is dually equivalent to $\Mdeve$.
\end{theorem}

\subsection{Additional properties of de Vries algebras and de Vries extensions}

We conclude this preliminary section by recalling some basic facts about de Vries algebras and de Vries extensions that we will use 
subsequently. We start with de Vries algebras.


\begin{lemma} \label{prop: closed}
Let $\dv{A}$ be a de Vries algebra and $Y_\dv{A}$ its dual compact Hausdorff space.
\begin{enumerate}
\item There is an isomorphism between the lattice of round filters of $\dv{A}$ (ordered by reverse inclusion) and the lattice of round ideals
(ordered by inclusion), given by $F \mapsto \{ a \in \dv{A} \mid \lnot a \in F\}$ for a round filter $F$ and
$I \mapsto \{ a \in \dv{A} \mid \lnot a \in I\}$ for a round ideal $I$.
\item There is an isomorphism between the lattice of round filters of $\dv{A}$ and the lattice of closed subsets of $Y_\dv{A}$,
given by
$$
F \mapsto  \bigcap \{ \zeta_\dv{A}(a) \mid a \in F\} \mbox{ and } C \mapsto  \{ a \in \dv{A} \mid C \subseteq \zeta_\dv{A}(a) \}.
$$
\item There is an isomorphism between the lattice of round ideals of $\dv{A}$ and the lattice of open subsets of $Y_\dv{A}$, given by
$$
I \mapsto  \bigcup \{ \zeta_\dv{A}(a) \mid a \in I\} \mbox{ and } U \mapsto  \{ a \in \dv{A} \mid \cl(\zeta_\dv{A}(a)) \subseteq U\}.
$$
\end{enumerate}
\end{lemma}

The proof of Lemma~\ref{prop: closed}(1) is straightforward, that of Lemma~\ref{prop: closed}(2) is given in
\cite[Thm.~1.3.12]{deV62}, and Lemma~\ref{prop: closed}(3) is proved similarly. The proof of the next lemma is also straightforward, 
and we skip it.

\begin{lemma} \label{lem: extra properties of morphisms}
Let $\rho : \dv{A} \to \dv{A}'$ be a de Vries morphism.
\begin{enumerate}
\item If $a \in \dv{A}$, then $\rho(a) \le \lnot\rho(\lnot a)$.
\item If $I$ is an ideal of $\dv{A}$, $a_1,\dots,a_n\in I$, $x \in \dv{A}'$, and $x \le \rho(a_1) \vee \cdots \vee \rho(a_n)$, then 
there is $b \in I$ $(b=a_1 \vee \cdots \vee a_n)$ such that $x \le \rho(b)$.
\end{enumerate}
\end{lemma}


The next two lemmas are about de Vries extensions.

\begin{lemma} \label{lem: 4.1}
\cite[Lem.~4.1]{BMO18a}
Let $e : X \to Y$ be a compactification of a completely regular space $X$. If $e^{-1} : \RO(Y) \to \wp(X)$ is the corresponding de Vries extension, 
then the image of $e^{-1}$ is $\RO(X)$ and $e^{-1}$ is a Boolean isomorphism from $\RO(Y)$ to $\RO(X)$.
\end{lemma}

\begin{lemma} \label{lem: 5.3&6.2}
Let $\alpha : \dv{A} \to \dv{B}$ be a de Vries extension.
\begin{enumerate}
\item \cite[Lem.~5.3]{BMO18a}
If $a \in \dv{A}$ and $b \in X_\dv{B}$, then $b \le \alpha(a)$ iff $\alpha_*(b) \in \zeta(a)$.
\item \cite[Lem.~6.2]{BMO18a}
If $\gamma : \dv{C} \to \dv{B}$ is another de Vries extension, then $\alpha$ and $\gamma$ are compatible iff they induce the same topology on 
$X_\dv{B}$.
\end{enumerate}
\end{lemma}

\section{Normal de Vries extensions}

In this section we introduce normal de Vries extensions and show that every normal de Vries extension is maximal. We prove that the dual equivalence of Theorem~\ref{thm: comp reg} between $\creg$ and $\Mdeve$ restricts to a dual equivalence between the full subcategory $\norm$ of $\creg$ consisting of normal spaces and the full subcategory $\ndeve$ of $\Mdeve$ consisting of normal de Vries extensions.

\begin{definition}
\begin{enumerate}
\item[]
\item We call a de Vries extension $\alpha : \dv{A} \to \dv{B}$ \emph{normal} provided the following axiom holds: If $F$ is a round filter and
$I$ a round ideal of $\dv{A}$ with $\bigwedge \alpha[F] \le \bigvee \alpha[I]$, then there are $a,b\in\dv{A}$ such that $a \prec b$,
$\bigwedge \alpha[F] \le \alpha(a)$, and $\alpha(b) \le \bigvee \alpha[I]$.
\item Let $\ndeve$ be the full subcategory of $\deve$ consisting of normal de Vries extensions.
\end{enumerate}
\end{definition}

It is a well-known theorem (see, e.g., \cite[Cor.~3.6.4]{Eng89}) that a compactification $e:X\to Y$ of a normal space is equivalent to the Stone-\v{C}ech compactification $s : X \to \beta X$ iff disjoint closed sets in $X$ have disjoint closures in $Y$. We next prove an algebraic version of this result by characterizing normal de Vries extensions.

\begin{theorem} \label{prop:normality}
Let $e : X \to Y$ be a compactification. Then the associated de Vries extension $e^{-1} : \RO(Y) \to \wp(X)$ is normal iff $X$ is normal and $e$ is isomorphic to the Stone-\v{C}ech compactification $s : X \to \beta X$ of $X$. In particular, $X$ is normal iff $s^{-1} : \RO(\beta X) \to \wp(X)$ is a normal de Vries extension.
\end{theorem}

\begin{proof}
First suppose that $X$ is normal and $e$ is isomorphic to the Stone-\v{C}ech compactification $s : X \to \beta X$. To simplify notation, we identify $e$ with $s$ and view $X$ as a dense subspace of $\beta X$. Then $s^{-1}(W) = W \cap X$. Let $F$ be a round filter and $I$ a round ideal of $\RO(\beta X)$ with $\bigwedge s^{-1}[F] \le \bigvee s^{-1}[I]$. We have $\bigwedge s^{-1}[F] = (\bigcap F) \cap X =: C$ is closed in $X$, $\bigvee s^{-1}[I] = (\bigcup I) \cap X =: U$ is open in $X$, and $C\subseteq U$. Let $D = X \setminus U$. Then $D$ is closed and $C \cap D = \varnothing$. Using that $X$ is normal twice, there are $V_1, V_2 \in \RO(X)$ with $C \subseteq V_1$, $D \subseteq V_2$, and $\cl_X(V_1) \cap \cl_X(V_2) = \varnothing$. By Lemma~\ref{lem: 4.1}, there are $W_1, W_2 \in \RO(\beta X)$ with $W_i \cap X = V_i$. By \cite[Cor.~3.6.4]{Eng89}, $\cl_{\beta X}(V_1) \cap \cl_{\beta X}(V_2) = \varnothing$, so since $X$ is dense in $\beta X$, we have $\cl_{\beta X}(W_1) \cap \cl_{\beta X}(W_2) = \varnothing$. Therefore, $\cl_{\beta X}(W_1) \subseteq \beta X \setminus \cl_{\beta X}(W_2)$. Since $W_2$ is regular open, $\cl_{\beta X}(W_2)$ is regular closed, so $W_3 := \beta X \setminus \cl_{\beta W}(W_2)$ is regular open and $W_1 \prec W_3$. We have $C \subseteq V_1 = W_1 \cap X$ and
\begin{align*}
W_3 \cap X &= X \cap (\beta X \setminus \cl_{\beta X}(W_2))= X \setminus \cl_{\beta X}(W_2) \\
&= X \setminus \cl_{\beta X}(W_2 \cap X) = X \setminus \cl_X(V_2) \subseteq X \setminus V_2 \subseteq U.
\end{align*}
Therefore, we have found $W_1 \prec W_3$ with $\bigwedge s^{-1}[F] \subseteq W_1$ and $W_3 \subseteq \bigvee s^{-1}[I]$. This proves that $s^{-1} : \RO(\beta X) \to \wp(X)$ is a normal de Vries extension.

Conversely, suppose that $e^{-1} : \RO(Y) \to \wp(X)$ is a normal de Vries extension. By \cite[Cor.~3.6.4]{Eng89}, to show that $X$ is normal and 
$e$ is isomorphic to the Stone-\v{C}ech compactification of $X$ it is sufficient to show that if $C$ and  $D$ are disjoint closed sets of $X$, then 
$\cl_Y(C) \cap \cl_Y(D) = \varnothing$. Let $U = X \setminus D$. Then $C \subseteq U$. Set $V = \int_Y(U \cup (Y \setminus X))$. By 
Lemma~\ref{prop: closed}, there is a round filter $F$ of $\RO(Y)$ with $\bigcap F = \cl_Y(C)$ and a round ideal $I$ of $\RO(Y)$ with $\bigcup I = V$. 
Therefore, $\bigwedge e^{-1}[F] = \cl_Y(C) \cap X = C$ and $\bigvee e^{-1}[I] = V \cap X =  U$. Since $e^{-1} : \RO(Y) \to \wp(X)$ is a normal 
de Vries extension, there are $W_1,W_2 \in \RO(Y)$ with $W_1 \prec W_2$, $C \subseteq W_1 \cap X$, and $W_2 \cap X \subseteq U$. Then 
$\cl_Y(W_1) \subseteq W_2$. Since $C \subseteq W_1 \cap X \subseteq W_1$, we see that $\cl_Y(C) \subseteq \cl_Y(W_1) \subseteq W_2$. Also, 
$W_2 \cap X \subseteq U$, so $W_2 \subseteq U \cup (Y \setminus X)$. Since $W_2$ is open in $Y$, it follows that 
$W_2 \subseteq \int_Y(U \cup (Y \setminus X)) = V$. Therefore, $\cl_Y(C) \subseteq V$, so $\cl_Y(C) \cap (Y\setminus V) = \varnothing$. Finally,
\[
Y \setminus V = Y \setminus \int_Y(U \cup (Y \setminus X)) = \cl_Y(X \setminus U) = \cl_Y(D).
\]
Consequently, $\cl_Y(C) \cap \cl_Y(D) = \varnothing$. This finishes the proof.
\end{proof}

\begin{corollary}
Let $\alpha : \dv{A} \to \dv{B}$ be a normal de Vries extension. Then $\alpha$ is maximal.
\end{corollary}

\begin{proof}
This follows from Theorems~\ref{thm: Stone-Cech} and \ref{prop:normality}
\end{proof}

\begin{remark}
Using Theorem~\ref{thm: duality}, we can phrase Theorem~\ref{prop:normality} dually as follows: A de Vries extension $\alpha : \dv{A} \to \dv{B}$ is normal iff $X_\dv{B}$ is normal and the corresponding compactification $\alpha_* : X_\dv{B} \to Y_\dv{A}$ is isomorphic to the Stone-\v{C}ech compactification of $X_\dv{B}$.
\end{remark}

Let $\norm$ be the full subcategory of $\creg$ consisting of normal spaces. Putting Theorems~\ref{thm: comp reg} and~\ref{prop:normality} together yields the following duality theorem for normal spaces.

\begin{theorem}
There is a dual equivalence between $\norm$ and $\ndeve$.
\end{theorem}

\section{Locally compact de Vries extensions}

In this section we introduce locally compact de Vries extensions. Unlike normal de Vries extensions, locally compact de Vries extensions do not have to be maximal. We prove that the category $\ldeve$ of locally compact maximal de Vries extensions is dually equivalent to the category $\LKHaus$ of locally compact Hausdorff spaces.

Let $e : X \to Y$ be a compactification and $e^{-1}:\RO(Y)\to\wp(X)$ the corresponding de Vries extension. For $U\in\RO(Y)$, we have:
\begin{align*}
\lnot e^{-1}(\lnot U) &=  X \setminus e^{-1}(\int_Y(Y \setminus U)) = X \setminus e^{-1}(Y \setminus \cl_Y(U)) \\
&=  X \setminus (X \setminus e^{-1}(\cl_Y(U))) = e^{-1}(\cl_Y(U)).
\end{align*}
Therefore, $\lnot e^{-1}(\lnot U)$ is a closed subset of $X$. For it to be compact, since $e^{-1}[\RO(Y)]=\RO(X)$ by Lemma~\ref{lem: 4.1}, if $\lnot e^{-1}(\lnot U) \subseteq \bigcup \{ e^{-1}(V_i) \mid i\in I \}$, with $V_i \in \RO(Y)$, then there is a finite $J \subseteq I$ with $\lnot e^{-1}(\lnot U) \subseteq \bigcup \{ e^{-1}(V_i) \mid i\in J \}$. This motivates the following definition.

\begin{definition}
Let $\alpha : \dv{A} \to \dv{B}$ be a de Vries extension.
\begin{enumerate}
\item We call $a \in \dv{A}$ \emph{$\alpha$-compact} provided $\lnot\alpha(\lnot a) \le \bigvee \alpha[S]$ for some $S \subseteq \dv{A}$ implies that there is a finite $T \subseteq S$ with $\lnot \alpha(\lnot a) \le \bigvee \alpha[T]$.
\item Let $I_\alpha = \{ a \in \dv{A} \mid a$ is $\alpha$-compact$\}$.
\end{enumerate}
\end{definition}

\begin{lemma}\label{lem: I_alpha}
If $\alpha: \dv{A} \to \dv{B}$ is a de Vries extension, then $I_\alpha$ is an ideal of $\dv{A}$.
\end{lemma}

\begin{proof}
First, $0 \in I_\alpha$ because $\lnot \alpha(\lnot 0) = \lnot\alpha(1) = \lnot 1 = 0$ is clearly $\alpha$-compact, so $I_\alpha$ is nonempty. Next, let $a \in I_\alpha$ and
$b \le a$. Suppose that $\lnot\alpha(\lnot b) \le \bigvee \alpha[S]$ for some $S \subseteq \dv{A}$. Then
\[
\lnot\alpha(\lnot a) \le 1 = \lnot\alpha(\lnot b) \vee \alpha(\lnot b) \le \bigvee \alpha[S] \vee \alpha(\lnot b),
\]
so there is a finite $T \subseteq S$ with $\lnot\alpha(\lnot a) \le \bigvee \alpha[T] \vee \alpha(\lnot b)$. Since $b \le a$, we have $\lnot\alpha(\lnot b) \le \lnot\alpha(\lnot a)$, so $\lnot\alpha(\lnot b) \le \bigvee \alpha[T]$, which shows that $b \in I_\alpha$. Finally, let $a,b \in I_\alpha$. Suppose $\lnot\alpha(\lnot(a \vee b)) \le \bigvee \alpha[S]$. Then $\lnot\alpha(\lnot a) \vee \lnot\alpha(\lnot b) = \lnot\alpha(\lnot(a \vee b)) \le \bigvee \alpha[S]$. Therefore, there are finite $T, T' \subseteq S$ with $\lnot\alpha(\lnot a) \le \bigvee \alpha[T]$ and $\lnot\alpha(\lnot b) \le \bigvee \alpha[T']$. Thus, $\lnot\alpha(\lnot(a \vee b)) \le \bigvee \alpha[T \cup T']$, yielding that $a \vee b \in I_\alpha$. This completes the proof that $I_\alpha$ is an ideal of $\dv{A}$.
\end{proof}

\begin{theorem} \label{prop: I_alpha}
Let $\alpha : \dv{A} \to \dv{B}$ be a de Vries extension. Then the following are equivalent.
\begin{enumerate}
\item $X_{\dv{B}}$ is locally compact.
\item For each $b \in \dv{A}$ we have $\alpha(b) = \bigvee \{ \alpha(a) \mid a \in I_\alpha, a \prec b \}$.
\item $I_\alpha$ is a round ideal with $\bigvee \alpha[I_\alpha] = 1$.
\end{enumerate}
\end{theorem}

\begin{proof}
(1)$\Rightarrow$(2). Suppose that $X_{\dv{B}}$ is locally compact. The de Vries extension $\alpha : \dv{A} \to \dv{B}$ is isomorphic to $\alpha_*^{-1} : \RO(Y_\dv{A}) \to \wp(X_{\dv{B}})$. Under this isomorphism, $I_\alpha$ corresponds to $I := I_{\alpha_*^{-1}}$. To simplify notation we view $X_\dv{B}$ as a dense subspace of $Y_\dv{A}$. Let $V \in \RO(Y_\dv{A})$. Then the map $\alpha_*^{-1}$ sends $V$ to $V \cap X_\dv{B}$, so $V \cap X_\dv{B}$ is the union of those $U\in\RO(X_{\dv{B}})$ for which $\cl_{X_{\dv{B}}}(U)\subseteq V \cap X_\dv{B}$ and is compact. Since $\cl_{X_{\dv{B}}}(U)$ is compact, $\cl_{X_{\dv{B}}}(U)$ is closed in $Y_\dv{A}$. Thus, $\cl_{Y_\dv{A}}(U) = \cl_{X_{\dv{B}}}(U)$. By Lemma~\ref{lem: 4.1}, $U = W \cap X_\dv{B}$ for some $W \in \RO(Y_\dv{A})$. Because $X_\dv{B}$ is dense in $Y_\dv{A}$, we have $\cl_{Y_\dv{A}}(W) = \cl_{Y_\dv{A}}(U)$, which yields $\cl_{Y_\dv{A}}(W) \subseteq V \cap X_\dv{B} \subseteq V$. Therefore, $W \prec V$. Moreover, as $\cl_{Y_\dv{A}}(W) \cap X_\dv{B} = \cl_{Y_\dv{A}}(W) = \cl_{X_\dv{B}}(U)$ is compact, $W \in I$. Thus, $V \cap X_\dv{B}$ is the union of $W \cap X_\dv{B}$ for $W \in I$ and $W \prec V$, which yields (2).

(2)$\Rightarrow$(3). By Lemma~\ref{lem: I_alpha}, $I_\alpha$ is an ideal. By (2), $\alpha(1) = \bigvee \{ \alpha(a) \mid a \in I_\alpha\}$, so 
$1 = \bigvee \alpha[I_\alpha]$. To see that $I_\alpha$ is round, let $a \in I_\alpha$. Then $\lnot\alpha(\lnot a) \le 1 = \bigvee \alpha[I_\alpha]$. 
Since $a$ is $\alpha$-compact, there are $a_1,\dots, a_n \in I_\alpha$ with $\lnot\alpha(\lnot a) \le \alpha(a_1) \vee \cdots \vee \alpha(a_n)$. 
By Lemma~\ref{lem: extra properties of morphisms}, there is $b \in I_\alpha$ with $\lnot\alpha(\lnot a) \le \alpha(b)$. By (2), 
$\alpha(b)=\bigvee\{\alpha(c)\mid c\in I_\alpha,c\prec b\}$, so repeating the above argument with $1$ replaced by $b$ yields 
$c \prec b$ with $\lnot\alpha(\lnot a) \le \alpha(c)$. Thus, $\alpha(a) \le \lnot\alpha(\lnot a) \le \alpha(c)$, so $a \le c$, 
and hence $a \prec b$. Since $b \in I_\alpha$, this shows that $I_\alpha$ is round.

(3)$\Rightarrow$(1). Since $I_\alpha$ is a round ideal, by Lemma~\ref{prop: closed}, it corresponds to the open subset 
$U := \bigcup \{\zeta(a) \mid a \in I_\alpha\}$ of $Y_\dv{A}$.

\begin{claim} \label{claim: X open}
If $U= \bigcup \{\zeta(a) \mid a \in I_\alpha\}$, then $U = \alpha_*[X_\dv{B}]$.
\end{claim}

\begin{proofclaim}
Let $b$ be an atom of $\dv{B}$. Since $\bigvee \alpha[I_\alpha] = 1$ and $b$ is an atom, there is $a \in I_\alpha$ with $b \le \alpha(a)$. 
By Lemma~\ref{lem: 5.3&6.2}(1), $\alpha_*(b) \in \zeta(a)$. Thus, $\alpha_*[X_\dv{B}] \subseteq U$. For the reverse inclusion, let $y \in U$. 
Then there is $c \in I_\alpha$ with $c \in y$. Suppose $\bigwedge \alpha(y) = 0$. Then $1 = \bigvee \{ \lnot \alpha(a) \mid a \in y\}$. 
Since $y$ is a round filter, $c\in y$ implies there is $a\in y$ with $a\prec c$, so $\lnot\alpha(\lnot a)\prec\alpha(c)$, and hence 
$\lnot\alpha(c)\le\alpha(\lnot a)$. Therefore, $1=\bigvee \{ \alpha(\lnot a) \mid a \in y\}$, and so 
$\lnot\alpha(\lnot c)\le \bigvee \{ \alpha(\lnot a) \mid a \in y\}$. As $c \in I_\alpha$ and $y$ is 
closed under finite meets, there is $a \in y$ with $\lnot\alpha(\lnot c) \le \alpha(\lnot a)$. Therefore, 
$\alpha(a) \le \lnot \alpha(\lnot a) \le \alpha(\lnot c)$, yielding $a \le \lnot c$. Thus, $a \wedge c = 0$, 
which is false since $a \wedge c \in y$. Consequently, $\bigwedge \alpha(y) \ne 0$, and hence there is an atom 
$b$ with $b \le \bigwedge \alpha(y)$. This implies that $y \subseteq \alpha^{-1}({\uparrow}b)$, and since $y$ is round, 
$y \subseteq \thu \alpha^{-1}({\uparrow}b) = \alpha_*(b)$. Because $y$ is an end, we have equality, and so 
$y \in \alpha_*[X_\dv{B}]$. This completes the proof that $\alpha_*[X_\dv{B}] = U$.
\end{proofclaim}

From the claim we see that $\alpha_*[X_\dv{B}]$ is open in $Y_\dv{A}$, which implies that $X_\dv{B}$ is locally compact.
\end{proof}

\begin{definition}
We call a de Vries extension $\alpha : \dv{A} \to \dv{B}$ \emph{locally compact} provided
\[
\alpha(b) = \bigvee \{ \alpha(a) \mid a\in I_\alpha, a \prec b \} \mbox{ for all } b \in \dv{A}.
\]
\end{definition}

\begin{remark}\label{rem: locally compact}
By Theorem~\ref{thm: duality}, we can phrase Theorem~\ref{prop: I_alpha} dually as follows: Let $e : X \to Y$ be a compactification and let $\alpha = e^{-1} : \RO(Y) \to \wp(X)$ be the corresponding de Vries extension. Then the following are equivalent.
\begin{enumerate}
\item $X$ is locally compact.
\item $\alpha$ is locally compact.
\item The ideal $I_{\alpha}$ is round and $\bigcup \alpha[I_{\alpha}] = X$.
\end{enumerate}
In particular, $X$ is locally compact iff $s^{-1} : \RO(\beta X) \to \wp(X)$ is a locally compact de Vries extension.
\end{remark}

We next show that not every locally compact de Vries extension is maximal.

\begin{example}
Let $X$ be the set of natural numbers equipped with the discrete topology, and let $c : X \to \omega X$ be the one-point compactification of $X$ \cite[Thm.~3.5.11]{Eng89}. By Remark~\ref{rem: locally compact}, the de Vries extension $c^{-1} : \RO(Y) \to \wp(X)$ is locally compact. However, since $c$ is not isomorphic to the Stone-\v{C}ech compactification of $X$, the de Vries extension $c^{-1}$ is not maximal by \cite[Thm.~6.4]{BMO18a}.
\end{example}

\begin{definition}
\begin{enumerate}
\item[]
\item Let $\ldeve$ be the full subcategory of $\Mdeve$ consisting of locally compact maximal de Vries extensions.
\item Let $\LKHaus$ be the full subcategory of $\creg$ consisting of locally compact spaces.
\end{enumerate}
\end{definition}

\begin{theorem}\label{thm: loc comp}
There is a dual equivalence between $\LKHaus$ and $\ldeve$.
\end{theorem}

\begin{proof}
Apply Theorems~\ref{thm: comp reg} and~\ref{prop: I_alpha}.
\end{proof}

Another duality for $\LKHaus$ was obtained by Dimov \cite{Dim10}. In Section~6 we will show how to derive Dimov's duality from Theorem~\ref{thm: loc comp}.

\section{Compact de Vries extensions}

In this section we introduce compact de Vries extensions and prove that the category of compact de Vries extensions is dually equivalent to 
the category of compact Hausdorff spaces. This yields that the category of compact de Vries extensions is equivalent to the category of de 
Vries algebras. We give a direct proof of this equivalence, thus providing a different perspective on de Vries duality. 

\begin{definition}
\begin{enumerate}
\item[]
\item We call a de Vries extension $\alpha : \dv{A} \to \dv{B}$ \emph{compact} provided the following axiom holds: If $F$ is a round filter and $I$ a round ideal of $\dv{A}$ with $\bigwedge \alpha[F] \le \bigvee \alpha[I]$, then $F \cap I \ne \varnothing$.
\item Let $\Cdeve$ be the full subcategory of $\deve$ consisting of compact de Vries extensions.
\end{enumerate}
\end{definition}

\begin{remark} \label{rem: compact implies normal and maximal}
Every compact de Vries extension is normal. To see this, let $\alpha : \dv{A} \to \dv{B}$ be a compact de Vries extension, $F$ a round filter, and $I$ a round ideal of $\dv{A}$ with $\bigwedge \alpha[F] \le \bigvee \alpha[I]$. Then there is $a \in F \cap I$. Since $I$ is round, there is $b \in I$ with $a \prec b$. Therefore, $\bigwedge \alpha[F] \le \alpha(a)$ and $\alpha(b) \le \bigvee \alpha[I]$. Thus, $\alpha$ is normal. Consequently, $\Cdeve$ is a full subcategory of $\ndeve$, and hence also of $\Mdeve$.
\end{remark}

\begin{theorem} \label{prop: equivalence of compact}
For a de Vries extension $\alpha : \dv{A} \to \dv{B}$, the following are equivalent.
\begin{enumerate}
\item $\alpha$ is a compact de Vries extension.
\item $X_{\dv{B}}$ is compact.
\item $I_\alpha = \dv{A}$.
\end{enumerate}
\end{theorem}

\begin{proof}
(1)$\Rightarrow$(2): Let $x$ be an end of $\dv{A}$. If $\bigwedge \alpha[x] = 0$, then for $F = x$ and $I = \{0\}$, (1) implies that $0 \in x$, which is false. Therefore, there is an atom $b$ of $\dv{B}$ with $b \le \bigwedge \alpha[x]$. This means $x \subseteq \alpha^{-1}({\uparrow}b)$, and so $x \subseteq \thu \alpha^{-1}({\uparrow}b)$. Since $\thu \alpha^{-1}({\uparrow}b)$ is a round filter and $x$ is an end, we obtain $x = \thu \alpha^{-1}({\uparrow}b)$, and so $x = \alpha_*(b)$. Thus, $\alpha_*[X_\dv{B}] = Y_\dv{A}$, and hence $X_\dv{B}$ is compact.

(2)$\Rightarrow$(3): Suppose that $X_\dv{B}$ is compact. Since $\alpha$ is isomorphic to $\alpha_*^{-1} : \RO(Y_\dv{A}) \to \wp(X_\dv{B})$ and $\lnot\alpha_*^{-1}(\lnot Y_\dv{A}) = X_\dv{B}$ is compact, we see that $1 \in I_\alpha$, so $I_\alpha = \dv{A}$.

(3)$\Rightarrow$(1): Suppose that $I_\alpha = \dv{A}$. Let $F$ be a round filter and $I$ a round ideal of $\dv{A}$ with $\bigwedge \alpha[F] \le \bigvee \alpha[I]$. Then $\lnot \bigwedge \alpha[F] \vee \bigvee \alpha[I] = 1$. We show that $F \cap I \ne \varnothing$. Let $J = \{ c \in \dv{A} \mid \lnot c \in F\}$. Since $F$ is a round filter, $J$ is a round ideal. We show that $\lnot\bigwedge\alpha[F] = \bigvee \alpha[J]$. Since $F$ is round, $b \in F$ implies that there is $a \in F$ with $a \prec b$. Then $\alpha(a) \le \lnot\alpha(\lnot a) \le \alpha(b)$. Therefore,
\begin{align*}
\lnot\bigwedge \alpha[F] &= \lnot\bigwedge \{ \lnot\alpha(\lnot a) \mid a \in F\} =  \bigvee \{ \alpha(\lnot a) \mid a \in F\} \\
&=  \bigvee \{ \alpha(c) \mid c \in J\} = \bigvee \alpha[J].
\end{align*}
From this equality we have $1 = \bigvee \alpha[J] \vee \bigvee \alpha[I]$. Since $1 \in I_\alpha$, there are $a_1,\dots, a_n \in J$ and $b_1, \dots, b_m \in I$ with $1 = \alpha(a_1) \vee \cdots \vee \alpha(a_n) \vee \alpha(b_1) \vee \cdots \vee \alpha(b_m)$. Let $a = a_1 \vee \cdots \vee a_n$ and $b = b_1 \vee \cdots \vee b_m$. Then $a \in J$, $b \in I$, $\alpha(a_1) \vee \cdots \vee \alpha(a_n) \le \alpha(a)$, and $\alpha(b_1) \vee \cdots \vee \alpha(b_m) \le \alpha(b)$. Therefore, $1 = \alpha(a) \vee \alpha(b)$, and so $\lnot\alpha(\lnot a) \vee \alpha(b) = 1$, which gives $\alpha(\lnot a) \le \alpha(b)$. This implies $\lnot a \le b$, so $\lnot a \in I$. Because $a \in J$, we have $\lnot a \in F$, and hence $\lnot a \in F \cap I$. Thus, $\alpha : \dv{A} \to \dv{B}$ is a compact de Vries extension.
\end{proof}

\begin{remark}\label{rem: compact}
By Theorem~\ref{thm: duality}, we can phrase Theorem~\ref{prop: equivalence of compact} dually as follows: Let $e : X \to Y$ be a compactification and let $\alpha = e^{-1} : \RO(Y) \to \wp(X)$ be the corresponding de Vries extension. Then the following are equivalent.
\begin{enumerate}
\item $X$ is compact.
\item $\alpha$ is compact.
\item $I_{\alpha} = \RO(Y)$.
\end{enumerate}
In particular, $X$ is compact iff $s^{-1} : \RO(\beta X) \to \wp(X)$ is a compact de Vries extension.
\end{remark}

Since $\KHaus$ is a full subcategory of $\creg$, we may interpret it as a full subcategory of $\C$. This interpretation sends a compact Hausdorff space $X$ to the compactification $ X \to X$.

\begin{theorem} \label{thm: compact}
$\Cdeve$ is dually equivalent to $\KHaus$.
\end{theorem}

\begin{proof}
By Remark~\ref{rem: compact implies normal and maximal}, $\Cdeve$ is a full subcategory of $\Mdeve$. The result then follows from Theorems~\ref{thm: comp reg} and~\ref{prop: equivalence of compact}.
\end{proof}

This together with de Vries duality yields that $\Cdeve$ is equivalent $\dev$. We give a direct proof of this result, which provides a different perspective on de Vries duality.

\begin{theorem} \label{thm: cdeve = dev}
$\Cdeve$ is equivalent to $\dev$.
\end{theorem}

\begin{proof}
Define a functor $\sf{D} : \dev \to \Cdeve$ as follows. If $\dv{A}$ is a de Vries algebra, then $\zeta_\dv{A} : \dv{A} \to \wp(Y_\dv{A})$ is a de Vries extension, which is compact by Theorem~\ref{prop: equivalence of compact}, and we set $\sf{D}(\dv{A}) = \zeta_\dv{A}$. To define $\sf{D}$ on morphisms, for a de Vries morphism $\rho : \dv{A} \to \dv{A'}$, consider the diagram
\[
\xymatrix@C5pc{
\dv{A} \ar[r]^{\zeta_\dv{A}} \ar[d]_{\rho} & \wp(Y_\dv{A}) \ar[d]^{\rho_*^{-1}} \\
\dv{A'} \ar[r]_{\zeta_{\dv{A'}}} & \wp(Y_{\dv{A'}}).
}
\]
For $a \in \dv{A}$, we have
\[
\rho_*^{-1}(\zeta_{\dv{A}}(a)) = \{ y \in Y_{\dv{A'}} \mid a \in \rho_*(y) \} =
\{ y \in Y_{\dv{A'}} \mid \exists c\prec a : \rho(c) \in y\}.
\]
Also, $(\zeta_{\dv{A'}}\star\rho)(a) = \bigvee \{ \zeta_{\dv{A'}}(\rho(c)) \mid c \prec a\}$. Since $\zeta_{\dv{A'}}(\rho(c)) = \{ y \in Y_{\dv{A'}} \mid \rho(c) \in y\}$ and $\bigvee$ in $\wp(Y_{\dv{A'}})$ is union, we see that $\rho_*^{-1} \circ \zeta_\dv{A} = \zeta_{\dv{A'}} \star \rho$. Therefore, $(\rho, \rho_*^{-1})$ is a morphism in $\deve$, and we set $\sf{D}(\rho) = (\rho, \rho_*^{-1})$. It is clear that $\sf{D}$ sends identity morphisms to identity morphisms. If $\rho : \dv{A} \to \dv{A'}$ and $\tau : \dv{A'} \to \dv{A''}$ are morphisms in $\dev$, then $\sf{D}(\tau\star\rho) = (\tau\star\rho, (\tau\star\rho)_*^{-1})$. Since $(\tau\star\rho)_*^{-1} = (\rho_*\circ\tau_*)^{-1} = \tau_*^{-1} \circ \rho_*^{-1}$, we see that $\sf{D}(\tau\star\rho) = (\tau\star\rho, \tau_*^{-1}\circ\rho_*^{-1}) = (\tau, \tau_*^{-1}) \circ (\rho, \rho_*^{-1})$. Thus, $\sf{D}$ is a functor.

Since $\sf{D}(\rho) = (\rho, \rho_*^{-1})$, it is clear that $\sf{D}$ is faithful. To see that $\sf{D}$ is full, let $(\rho, \sigma)$ be a morphism between $\zeta_\dv{A} : \dv{A} \to \wp(Y_\dv{A})$ and $\zeta_\dv{A'} : \dv{A'} \to \wp(Y_\dv{A'})$. Then $\sigma\circ\zeta_\dv{A} = \zeta_\dv{A'}\star\rho$. If $a \in \dv{A}$, then as we saw above, $\sigma(\zeta_\dv{A}(a)) = (\zeta_{\dv{A}'} \star \rho)(a) = \rho_*^{-1}(\zeta_\dv{A}(a))$. Therefore, since $\zeta_{\dv{A}}[\dv{A}]$ is join-meet dense in $\wp(Y_\dv{A})$ and $\sigma, \rho_*^{-1}$ are both complete Boolean homomorphisms, we conclude that $\sigma = \rho_*^{-1}$. Thus, $(\rho, \sigma) = (\rho, \rho_*^{-1}) = \sf{D}(\rho)$, and hence $\sf{D}$ is full.

Let $\alpha : \dv{A} \to \dv{B}$ be a compact extension. Then $X_\dv{B}$ is compact by Theorem~\ref{prop: equivalence of compact}, so $\alpha_* : X_\dv{B} \to Y_\dv{A}$ is a homeomorphism. Therefore, $\alpha$ is isomorphic to $\zeta_\dv{A} : \dv{A} \to \wp(Y_\dv{A})$. Thus, $\sf{D} : \dev \to \Cdeve$ is an equivalence of categories \cite[Thm.~IV.4.1]{Mac71}.
\end{proof}

\section{One-point Compactifications and minimal de Vries extensions}


In this section we give an algebraic description of the one-point compactification of a non-compact locally compact Hausdorff space 
by introducing the concept of a minimal de Vries extension. 

As we pointed out in Remark~\ref{rem: equiv versus iso}, a compactification 
$e : X \to Y$ is equivalent to the Stone-\v{C}ech compactification $s : X \to \beta X$ iff $e$ is isomorphic to $s$ in $\C$. We next show 
that if $X$ is non-compact locally compact, then a corresponding result holds for the one-point compactification of $X$.

\begin{lemma} \label{lem: equivalent to alpha}
Let $X$ be a non-compact locally compact Hausdorff space and let $e : X \to Y$ be a compactification. If $e$ is isomorphic to the one-point compactification $c : X \to \omega X$ in $\C$, then $e$ and $c$ are equivalent.
\end{lemma}

\begin{proof}
By hypothesis, there is an isomorphism $(f,g)$ between $e$ and $c$, which means the following diagram is commutative.
\[
\xymatrix{
X \ar[r]^e \ar[d]_{f} & Y \ar[d]^{g} \\
X \ar[r]_{c} & \omega X
}
\]
Write $\omega X = c(X) \cup \{\infty\}$. We claim that $|Y \setminus e[X]| = 1$. Since $g$ is onto, there is $w \in Y$ with $g(w) = \infty$. Then $w \notin e[X]$ since $\infty \notin c[X]$. Let $y \in Y$ with $y \ne w$. Then $g(y) \ne \infty$ since $g$ is 1-1. Therefore, there is $x \in X$ with $g(y) = c(x)$. Since $f$ is onto, there is $x' \in X$ with $x = f(x')$. Therefore, $g(y) = c(f(x')) = g(e(x'))$. Since $g$ is 1-1, $y = e(x')$. This proves $Y\setminus e[X] = \{w\}$, yielding the claim. Define $h : Y \to \omega X$ by $h(e(x)) = c(x)$ if $x \in X$ and $h(w) = \infty$. This is well defined since $e$ is 1-1. Note that $h\circ e = c$ follows immediately from the definition. We prove that $h$ is a homeomorphism.  Let $V$ be open in $\omega X$. First suppose that $\infty \notin V$. Then $V$ is open in $c[X]$, so $V = c[U]$ for some open set $U$ of $X$. As $h$ is 1-1 and $h\circ e = c$, we see that that $h^{-1}(V) = h^{-1}(c[U]) = e[U]$. This is open in $Y$ since $U$ is open in $X$ and $e[X]$ is open in $Y$. Next suppose that $\infty \in V$. Then $V = \{\infty\} \cup c[U]$ with $U$ open in $X$ and $X \setminus U$ compact. By the previous case, we see that $h^{-1}(V) = \{w\} \cup e[U]$. Furthermore,
\[
Y \setminus h^{-1}(V) = e[X] \setminus e[U] = e[X \setminus U].
\]
Since $X \setminus U$ is compact, $e[X \setminus U]$ is compact, and so it is closed in $Y$. Thus, $h^{-1}(V)$ is open in $Y$. This completes the proof that $h$ is continuous. Because it is a bijection between compact Hausdorff spaces, it is a homeomorphism. Thus, $e : X \to Y$ and $c : X \to \omega X$ are equivalent as compactifications of $X$.
\end{proof}

\begin{definition}
We say that a de Vries extension $\alpha : \dv{A} \to \dv{B}$ is \emph{minimal} provided for every compatible de Vries extension $\gamma : \dv{C} \to \dv{B}$, there is a de Vries morphism $\delta: \dv{A} \to \dv{C}$ such that $\gamma \star \delta = \alpha$.
\[
\xymatrix{
\dv{A} \ar[dr]_{\delta} \ar[rr]^{\alpha} && \dv{B} \\
& \dv{C} \ar[ru]_{\gamma} &
}
\]
\end{definition}

\begin{theorem} \label{prop: minimal}
For a de Vries extension $\alpha : \dv{A} \to \dv{B}$, the following are equivalent.
\begin{enumerate}
\item $\alpha$ is minimal but not compact.
\item $X_\dv{B}$ is non-compact locally compact and $\alpha_*:X_\dv{B}\to Y_\dv{A}$ is isomorphic to the one-point compactification of $X_\dv{B}$.
\item $X_\dv{B}$ is non-compact locally compact and $\alpha_*$ is equivalent to the one-point compactification of $X_\dv{B}$.
\item $I_\alpha$ is an end ideal with $\bigvee \alpha[I_\alpha] = 1$.
\end{enumerate}
\end{theorem}

\begin{proof}
(1)$\Rightarrow$(3). Since $\alpha$ is not compact, $X_\dv{B}$ is non-compact by Theorem~\ref{prop: equivalence of compact}. Let 
$e : X_\dv{B} \to Z$ be a compactification. Then $e^{-1} : \RO(Z) \to \wp(X_\dv{B})$ is a de Vries extension. As 
$\zeta_\dv{B}: \dv{B} \to \wp(X_\dv{B})$ is a Boolean isomorphism, $\gamma := \zeta_\dv{B}^{-1} \circ e^{-1} : \RO(Z) \to \dv{B}$ 
is a de Vries extension, which is compatible with $\alpha$ by Lemma~\ref{lem: 5.3&6.2}(2). Since $\alpha$ is minimal, there is a 
de Vries morphism $\delta : \dv{A} \to \RO(Z)$ with $\gamma \star \delta = \alpha$. By de Vries duality, $\delta_*$ induces a 
continuous map $Z \to Y_\dv{A}$ making the following diagram commute.
\[
\xymatrix{
X_\dv{B} \ar[rr]^{e} \ar[dr]_{\alpha_*} && Z \ar[dl] \\
& Y_\dv{A} &
}
\]
Therefore, $\alpha_*:X_\dv{B}\to Y_\dv{A}$ is the least compactification of $X_\dv{B}$. Thus, $X_\dv{B}$ is locally compact and $\alpha_*:X_\dv{B}\to Y_\dv{A}$ is equivalent to the one-point compactification of $X_\dv{B}$ by \cite[Thm.~3.5.12]{Eng89}..

(2)$\Leftrightarrow$(3). The implication $\Rightarrow$ follows by Lemma~\ref{lem: equivalent to alpha} and the implication $\Leftarrow$ is clear.

(2)$\Rightarrow$(4). By Theorems~\ref{prop: I_alpha} and~\ref{prop: equivalence of compact}, $I_\alpha$ is a proper round ideal with $\bigvee \alpha[I_\alpha] = 1$. Claim~\ref{claim: X open} shows that the open subset of $Y_\dv{A}$ corresponding to $I_\alpha$ is $\alpha_*[X_\dv{B}]$. Because this set is the complement of a single point, we see that $I_\alpha$ is an end ideal.

(4)$\Rightarrow$(2). Since $I_\alpha$ is an end ideal, Claim~\ref{claim: X open} shows that $Y_\dv{A} \setminus \alpha_*[X_\dv{B}]$ is a single point. Therefore, $X_\dv{B}$ is non-compact locally compact and $\alpha_*:X_\dv{B}\to Y_\dv{A}$ is isomorphic to the one-point compactification of $X_\dv{B}$.

(3)$\Rightarrow$(1). By Theorem~\ref{prop: equivalence of compact}, $\alpha$ is not compact. Let $\gamma : \dv{C} \to \dv{B}$ be a compatible extension. Then the topology on $X_\dv{B}$ induced by $\gamma$ is the same as that induced by $\alpha$, and by \cite[Thm.~3.5.11]{Eng89} there is a continuous map $f : Y_\dv{C} \to Y_\dv{A}$ making the following diagram commute.
\[
\xymatrix{
X_\dv{B} \ar[rr]^{\gamma_*} \ar[dr]_{\alpha_*} && Y_\dv{C} \ar[ld]^f \\
& Y_\dv{A} &
}
\]
By de Vries duality, there is a de Vries morphism $\delta : \dv{A} \to \dv{C}$ with $\delta_* = f$. Since the functor $\sf{E} : \C \to \deve$ is faithful, $\gamma\star\delta = \alpha$, and hence $\alpha$ is a minimal de Vries extension.
\end{proof}

\begin{remark}
Using Theorem~\ref{thm: duality}, we can phrase Theorem~\ref{prop: minimal} as follows: Let $e : X \to Y$ be a compactification and let $\alpha = e^{-1} : \RO(Y) \to \wp(X)$ be the corresponding de Vries extension. Then the following are equivalent.
\begin{enumerate}
\item $X$ is non-compact locally compact and $e:X \to Y$ is isomorphic to the one-point compactification of $X$.
\item $X$ is non-compact locally compact and $e:X \to Y$ is equivalent to the one-point compactification of $X$.
\item $\alpha$ is minimal but not compact.
\item $I_{\alpha}$ is an end ideal with $\bigcup \alpha[I_{\alpha}] = X$.
\end{enumerate}
\end{remark}

\begin{remark}
Let $\alpha : \dv{A} \to \dv{B}$ be a de Vries extension.
\begin{enumerate}
\item If $\alpha : \dv{A} \to \dv{B}$ is compact, then every compatible de Vries extension is isomorphic to $\alpha$. To see this, let 
$\gamma : \dv{C} \to \dv{B}$ be compatible with $\alpha$. By Lemma~\ref{lem: 5.3&6.2}(2), the topology on $X_\dv{B}$ inherited from $\alpha$ 
is the same as that inherited from $\gamma$. Because $\alpha$ is compact, the space $X_\dv{B}$ is compact by 
Theorem~\ref{prop: equivalence of compact}. Consequently, the embeddings $\alpha_* : X_\dv{B} \to Y_\dv{A}$ and 
$\gamma_* : X_\dv{B} \to Y_\dv{C}$ are both homeomorphisms, and so $\alpha_*$ and $\gamma_*$ are isomorphic in $\C$ 
as we see from the following diagram.
\[
\xymatrix@C5pc{
X_\dv{B} \ar[r]^{\alpha_*} \ar@{=}[d] & Y_\dv{A} \ar[d]^{\gamma_*\circ\alpha_*^{-1}} \\
X_\dv{B} \ar[r]_{\gamma_*} & Y_\dv{C}
}
\]
By Theorem~\ref{thm: duality}, it follows that $\alpha$ and $\gamma$ are isomorphic.

\item We show that $\alpha$ is both maximal and minimal iff $X_\dv{B}$ is almost compact, where we recall (see, e.g., \cite[p.~95]{GJ60}) that a completely regular space $X$ is \emph{almost compact} provided $|\beta X \setminus X| \le 1$. First suppose that $\alpha$ is not compact. By Theorem~\ref{prop: equivalence of compact}, $X_\dv{B}$ is non-compact. Therefore, by \cite[Thm.~6.4]{BMO18a} and Theorem~\ref{prop: minimal}, $\alpha$ is both maximal and minimal iff $X_\dv{B}$ is almost compact. Next suppose that $\alpha$ is compact. Then $X_\dv{B}$ is compact. Also, by (1), each compatible de Vries extension is isomorphic to $\alpha$, which implies that $\alpha$ is both maximal and minimal. Consequently, $\alpha$ is both maximal and minimal iff $X_\dv{B}$ is almost compact.
\end{enumerate}
\end{remark}

\section{Dimov duality for $\LKHaus$}

In Theorem~\ref{thm: loc comp} we proved that $\LKHaus$ is dually equivalent to $\ldeve$. In \cite[Thm.~3.12]{Dim10} Dimov proved that $\LKHaus$ is dually equivalent to a category we denote by $\Dim$ below. It follows that there is an equivalence between
$\ldeve$ and $\Dim$.
\[
\xymatrix{
& \LKHaus \ar@{<->}[dl] \ar@{<->}[dr] & \\
\Dim \ar@{<->}[rr] && \ldeve
}
\]
In this section we give a direct proof for why $\ldeve$ and $\Dim$ are equivalent, thus obtaining Dimov duality as a consequence
of Theorem~\ref{thm: loc comp}.

\begin{definition}
A \emph{Dimov algebra} is a triple $\mathfrak{D} = (A, \lhd, I)$, where $A$ is a complete Boolean algebra, $\lhd $ is a binary relation on
$A$ satisfying (DV1)--(DV5) of Definition~\ref{def: DeV}, and $I$ is an ideal of $A$ satisfying
\begin{enumerate}
\item[(I1)]If $a \in I$ and $a \lhd b$, then there is $c \in I$ with $a \lhd c \lhd b$.
\item[(I2)]If $(a \wedge c) \lhd (b\vee \lnot c)$ for all $c \in I$, then $a \lhd b$.
\item[(I3)]If $b \ne 0$, then there is $0 \ne a \in I$ with $a \lhd b$.
\end{enumerate}
\end{definition}

\begin{remark} \label{rem: Dimov alternative}
In \cite{Dim10} Dimov worked with contact relations $\delta$ and the resulting contact algebras. The two relations $\delta$ and $\lhd$ are dual to each other: $a \delta b$ iff $a \not\lhd \lnot b$. The axioms in terms of $\delta$ are given in \cite[Def.~2.1]{Dim10}, and it is straightforward to see that they are equivalent to axioms (DV1)-(DV5) for $\lhd$. Axioms (I1) and (I3) are the same as the axioms Dimov gives in \cite[Def.~2.9]{Dim10}. Our axiom (I2) is slightly different from the corresponding axiom of Dimov, which can be phrased in the language of $\lhd$ as follows:
\begin{equation}
\mbox{If } a \lhd (b \vee \lnot c) \mbox{ for all } c \in I, \mbox{ then } a \lhd b. \tag{$*$}
\end{equation}
Clearly (I2) implies $(*)$. For the converse, suppose that $(*)$ holds and $(a \wedge c) \lhd (b\vee \lnot c)$ for all $c \in I$. Let $d \in I$. For each $c \in I$, we have $a \wedge d \le (a \wedge d) \vee (a \wedge c) = a \wedge (c \vee d)$. Since $c \vee d \in I$, by assumption, $a \wedge (c \vee d) \lhd b \vee \lnot (c \vee d)$. Since $b \vee \lnot (c \vee d) \le b \vee \lnot c$, we have $a \wedge d \lhd b \vee \lnot c$ for all $c \in I$ by (DV3). Thus, $(*)$ implies $(a \wedge d) \lhd b$. Since this is true for all $d \in I$, by (DV5), $\lnot b \lhd (\lnot a\vee\lnot d)$ for all $d \in I$. Applying $(*)$ again yields $\lnot b\lhd\lnot a$, so $a \lhd b$ by (DV5). This shows that we can replace $(*)$ by (I2) above.
\end{remark}

\begin{definition}
A map $\rho : \mathfrak{D} \to \mathfrak{D'}$ between Dimov algebras is a \emph{Dimov morphism} if $\rho$ satisfies axioms
(M1) and (M2) of Definition~\ref{def: DeV} together with
\begin{enumerate}
\item[(D3)]If $a \lhd b$, then $\lnot\rho(\lnot a) \lhd \rho(b)$.
\item[(D4)]If $c \in I'$, then there is $a \in I$ with $c \le \rho(a)$.
\item[(D5)] $\rho(b) = \bigvee \{ \rho(a) \mid a \in I, a \lhd b\}$.
\end{enumerate}
\end{definition}

Similar to the de Vries setting, the composition of two Dimov morphisms $\rho_1$ and $\rho_2$ is given by
\[
(\rho_2 \diamond \rho_1)(b)=\bigvee\{\rho_2(\rho_1(a))\mid  a \in I, a\lhd b\}.
\]
Like $\dev$, with this composition, Dimov algebras and Dimov morphisms form a category (see \cite[Prop.~4.24]{Dim10}), which we denote by $\Dim$.

\begin{remark} \label{rem: Dimov def}
In Dimov's original definition \cite[Def.~3.8]{Dim10}, a weaker version of Axiom (D3) is used:
\[
\mbox{If } a \in I \mbox{ and } a \lhd b, \mbox{ then } \lnot\rho(\lnot a) \lhd \rho(b). \tag{$**$}
\]
But it follows from \cite[Lem~4.19]{Dim10} that the two axioms are equivalent. Dimov's proof isn't point-free; we give an alternative, pointfree proof. Suppose ($**$) holds. To see that (D3) holds, let $a \lhd b$. To show that $\lnot \rho(\lnot a) \lhd \rho (b)$, by (I2) it is sufficient to prove that $\lnot \rho(\lnot a) \wedge c \lhd \rho (b) \vee \lnot c$ for each $c \in I'$. Let $c \in I'$. By (D4), there is $d \in I$ with $c \le \rho (d)$, and so $\lnot\rho(\lnot a) \wedge c \le \lnot\rho(\lnot a) \wedge \rho(d)$.
\begin{claim}\label{claim: technical fact}
If $\rho : A \to A'$ is a meet preserving function between Boolean algebras, then
$\lnot\rho(\lnot a) \wedge \rho(d) \le \lnot\rho(\lnot(a\wedge d))$ for all $a,d \in A$.
\end{claim}
\begin{proofclaim}
We have
\begin{align*}
d\wedge\lnot a\le\lnot a & \Rightarrow \rho(d\wedge\lnot a)\le\rho(\lnot a)  \Rightarrow \rho(d\wedge\lnot(a\wedge d))\le\rho(\lnot a)  \\
&\Rightarrow \rho(d)\wedge\rho(\lnot(a\wedge d))\le\rho(\lnot a)  \Rightarrow \lnot\rho(\lnot a)\wedge\rho(d)\wedge\rho(\lnot(a\wedge d))=0 \\
& \Rightarrow \lnot\rho(\lnot a) \wedge \rho(d) \le \lnot\rho(\lnot(a\wedge d)).
\end{align*}
Thus, the claim holds.
\end{proofclaim}

By Claim~\ref{claim: technical fact}, $\lnot \rho (\lnot a) \wedge \rho (d) \le \lnot \rho (\lnot (a \wedge d))$. But since $d \in I$, we have $a \wedge d \in I$, so ($**$) yields $\lnot \rho (\lnot (a \wedge d)) \lhd \rho (b)$. So
\[
\lnot \rho (\lnot a) \wedge c \le \lnot \rho (\lnot a) \wedge \rho (d) \le \lnot \rho (\lnot (a \wedge d)) \lhd \rho (b) \le \rho (b) \vee \lnot c
\]
for each $c \in I'$. Thus, $\lnot \rho (\lnot a) \lhd \rho (b)$.
\end{remark}

\begin{lemma} \label{lem: lhd facts}
Suppose that $\alpha : \dv{A} \to \dv{B}$ is a locally compact de Vries extension. Define $\lhd$ on $\dv{A}$ by
\[
a \lhd b \mbox{ iff } (a \wedge c) \prec (b \vee \lnot c) \mbox{ for all } c \in I_\alpha.
\]
\begin{enumerate}
\item If $a \prec b$, then $a \lhd b$.
\item If $a \lhd b$, then $\lnot\alpha(\lnot a) \le \alpha(b)$.
\item If $a \in I_\alpha$ and $a \lhd b$, then $a \prec b$.
\end{enumerate}
\end{lemma}

\begin{proof}
(1). If $c \in I_\alpha$, then $(a \wedge c) \le a \prec b \le b \vee \lnot c$, so $(a \wedge c) \prec (b \vee \lnot c)$, and hence $a \lhd b$.

(2). For each $c \in I_\alpha$ we have $(a \wedge c) \prec (b \vee \lnot c)$, so $\lnot\alpha(\lnot(a\wedge c)) \le \alpha(b \vee \lnot c)$.  By Claim~\ref{claim: technical fact}, we have $\lnot\alpha(\lnot a) \wedge \alpha(c) \le \lnot\alpha(\lnot(a\wedge c))$. We show that $\alpha(b \vee \lnot c) \le \alpha(b) \vee \lnot \alpha(c)$. To see this, $b\wedge c\le b$, so we obtain:
\begin{align*}
b\wedge c\le b & \Rightarrow \alpha(b\wedge c)\le\alpha(b)  \Rightarrow \alpha((b\vee\lnot c)\wedge c)\le\alpha(b)  \\
&\Rightarrow \alpha(b\vee\lnot c)\wedge\alpha(c)\le\alpha(b)  \Rightarrow \alpha(b\vee\lnot c)\le\alpha(b)\vee\lnot\alpha(c).
\end{align*}
Therefore, $\lnot\alpha(\lnot a) \wedge \alpha(c) \le \alpha(b) \vee \lnot \alpha(c)$. Thus, $\lnot\alpha(\lnot a) \wedge \lnot\alpha(b) \wedge \alpha(c) = 0$ for all $c\in I_\alpha$. Since $\alpha$ is locally compact, $\bigvee \{ \alpha(c) \mid c \in I_\alpha \} = 1$ by Theorem~\ref{prop: I_alpha}. Consequently,
\begin{align*}
0 &= \bigvee \{ \lnot\alpha(\lnot a) \wedge \lnot\alpha(b) \wedge \alpha(c) \mid c \in I_\alpha\} \\
&= \lnot\alpha(\lnot a) \wedge \lnot\alpha(b) \wedge \bigvee \{ \alpha(c) \mid c \in I_\alpha\} \\
&= \lnot\alpha(\lnot a) \wedge \lnot\alpha(b).
\end{align*}
Therefore, $\lnot\alpha(\lnot a) \wedge \lnot\alpha(b) = 0$, and so $\lnot\alpha(\lnot a) \le \alpha(b)$.

(3). By (2), $\lnot\alpha(\lnot a) \le \alpha(b)$. Since $\alpha$ is locally compact, 
$\alpha(b) = \bigvee \{ \alpha(c) \mid c \in I_\alpha, c \prec b\}$. Because $a$ is $\alpha$-compact, 
applying Lemma~\ref{lem: extra properties of morphisms} to the ideal $I_\alpha \cap \thd b$ yields 
$c \in I_\alpha$ with $\lnot\alpha(\lnot a) \le \alpha(c)$ and $c \prec b$. As $\alpha(a) \le \lnot\alpha(\lnot a)$ 
and $\alpha$ is an embedding, we have $a \le c$, so $a \prec b$.
\end{proof}

We define $\sf{D} : \ldeve \to \Dim$ as follows. For a locally compact de Vries extension $\alpha : \dv{A} \to \dv{B}$, let ${\sf D}(\alpha)=(A, \lhd, I_\alpha)$, where $A$ is the underlying complete Boolean algebra of $\dv{A}$ and $\lhd$ is defined by $a \lhd b$ iff $(a \wedge c) \prec (b \vee \lnot c)$ for all $c \in I_\alpha$ as in the statement of Lemma~\ref{lem: lhd facts}; and if $(\rho, \sigma)$ is a morphism in $\ldeve$, let $\sf{D}(\rho, \sigma) = \rho$.

\begin{proposition}\label{lem: ldeve to Dim}
$\sf{D} : \ldeve \to \Dim$ is a covariant functor.
\end{proposition}

\begin{proof}
Let $\alpha : \dv{A} \to \dv{B}$ be a locally compact de Vries extension. We show that $\sf{D}(\alpha) \in \Dim$. For this, we first
show that $\lhd$ satisfies (DV1)-(DV5).

(DV1) is clear.

(DV2): If $a \lhd b$, then $(a \wedge c) \prec (b \vee \lnot c)$ for all $c \in I_\alpha$. This implies $(a \wedge c) \le (b \vee \lnot c)$,
and so $a \wedge \lnot b \wedge c = 0$ for all $c\in I_\alpha$. Since $\alpha$ is locally compact, $\bigvee\alpha[I_\alpha]=1$. Because $\alpha$
is an embedding, the last equality implies $\bigvee I_\alpha = 1$. Therefore, $a \wedge \lnot b = 0$, so $a \le b$.

(DV3): If $a \le b \lhd c \le d$, then for all $e \in I_\alpha$ we have $a \wedge e \le b \wedge e \prec c \vee \lnot e \le d \vee \lnot e$.
Thus, $a \lhd d$.

(DV4): Suppose that $a \lhd b,c$. Then $(a \wedge e)\prec (b \vee \lnot e),(c \vee \lnot e)$ for all $e \in I_\alpha$. Thus,
$(a \wedge e) \prec ((b\wedge c) \vee \lnot e)$ for all $e \in I_\alpha$, so $a \lhd (b\wedge c)$.

(DV5): If $a \lhd b$, then $(a \wedge c) \prec (b \vee \lnot c)$ for all $c \in I_\alpha$. Therefore,
$(\lnot b \wedge c) \prec (\lnot a \vee \lnot c)$ for all $c \in I_\alpha$. Thus, $\lnot b \lhd \lnot a$.

We next show that $I_\alpha$ satisfies (I1)-(I3).

(I1): Let $a \in I_\alpha$ and $a \lhd b$. Then $\lnot \alpha(\lnot a) \le \alpha(b)$ by Lemma~\ref{lem: lhd facts}(2). Since $\alpha$ is locally compact, we have $\alpha(b) = \bigvee \{ \alpha(c) \mid c \in I_\alpha, c \prec b\}$. Because $a$ is $\alpha$-compact, applying Lemma~\ref{lem: extra properties of morphisms} to $I_\alpha \cap \thd b$ yields $c\in I_\alpha$ with $c \prec b$ and $\lnot\alpha(\lnot a) \le \alpha(c)$. Repeating this argument with $b$ replaced by $c$ yields $d \in I_\alpha$ with $d\prec c$ and $\lnot\alpha(\lnot a) \le \alpha(d)$. Since $\alpha(a) \le \lnot\alpha(\lnot a)$, we see that $a \le d \prec c$, so $a \prec c\prec b$. Thus, $a \lhd c \lhd b$ by Lemma~\ref{lem: lhd facts}(1).

(I2): If $(a\wedge c) \lhd (b \vee \lnot c)$ for all $c \in I_\alpha$, then $a \wedge c \in I_\alpha$, so $(a \wedge c) \prec (b \vee \lnot c)$ for all $c \in I_\alpha$ by Lemma~\ref{lem: lhd facts}(3). Thus, $a \lhd b$.

(I3): Suppose $b \ne 0$. Then $\alpha(b) \ne 0$ since $\alpha$ is an embedding. Therefore, as $\alpha(b) = \bigvee \{\alpha(a) \mid a \in I_\alpha, a \prec b\}$, there is $0 \ne a \in I_\alpha$ with $a \prec b$. By Lemma~\ref{lem: lhd facts}(1), $a \lhd b$.

This shows that $\sf{D}(\alpha) \in \Dim$, and hence $\sf{D}$ is well defined on objects. To see that $\sf{D}$ is well defined on morphisms, let $(\rho, \sigma)$ be a morphism between locally compact de Vries extensions $\alpha : \dv{A} \to \dv{B}$ and $\alpha' : \dv{A'} \to \dv{B'}$.
\[
\xymatrix@C5pc{
\dv{A} \ar[r]^{\alpha} \ar[d]_{\rho} & \dv{B} \ar[d]^{\sigma} \\
\dv{A'} \ar[r]_{\alpha'} & \dv{B'}
}
\]
We show that $\rho : \sf{D}(\alpha) \to \sf{D}(\alpha')$ is a morphism in $\Dim$. Since $\rho$ is a de Vries morphism, it satisfies (M1) and (M2). Suppose $a\in I_\alpha$ with $a \lhd b$. Then $a \prec b$ by Lemma~\ref{lem: lhd facts}(3), so $\lnot\rho(\lnot a) \prec \rho(b)$. This implies $\lnot\rho(\lnot a) \lhd \rho(b)$ by Lemma~\ref{lem: lhd facts}(1), so (D3) holds.

To verify (D4) we point out that if $b \in I_{\alpha}$, then
\[
\sigma(\alpha(b)) = (\alpha'\star\rho)(b) = \bigvee \{ \alpha'(\rho(a))\mid a \in I_\alpha, a \prec b\}.
\]
Therefore, $\bigvee \{ \sigma(\alpha(b)) \mid b \in I_\alpha \} = \bigvee \{ \alpha'(\rho(a)) \mid a \in I_\alpha \}$. But
\[
\bigvee \{ \sigma(\alpha(b)) \mid b \in I_\alpha \} = \sigma\left( \bigvee \{ \alpha(b) \mid b \in I_\alpha\} \right) = \sigma(1) = 1.
\]
So, if $c \in I_{\alpha'}$, then $\lnot\alpha'(\lnot c) \le \bigvee \{ \alpha'(\rho(a)) \mid a \in I_\alpha \}$. Since $c$ is $\alpha'$-compact and $I_\alpha$ is an ideal, there is $a \in I_\alpha$ with $\lnot\alpha'(\lnot c) \le \alpha'(\rho(a))$. As $\alpha'$ is an embedding and $\alpha'(c) \le \lnot \alpha'(\lnot c)$, we conclude that $c \le \rho(a)$. Thus, (D4) holds.

For (D5), since $\bigvee I_{\alpha'} = 1$, (D4) implies that $1 = \bigvee \{ \rho(a) \mid a \in I_\alpha\}$. Let $b \in \dv{A}$. Then
\[
\rho(b) = \rho(b) \wedge \bigvee \{ \rho(a) \mid a \in I_\alpha\} = \bigvee \{ \rho(a \wedge b) \mid a \in I_\alpha \} = \bigvee \{\rho(c) \mid c \in I_\alpha, c \le b\}.
\]
Now, if $c \in I_\alpha$ with $c \le b$, then $\rho(c) = \bigvee \{ \rho(a) \mid a \prec c\}$ by (M4). For $a$ with $a \prec c$, we have $a \in I_\alpha$ and $a \lhd c$ by Lemma~\ref{lem: lhd facts}(1). Thus, $\rho(b) = \bigvee \{ \rho(a) \mid a \in I_\alpha, a \lhd b\}$, and so (D5) holds.

It follows that $\sf{D}(\rho, \sigma) = \rho$ is a morphism in $\Dim$. It is clear that $\sf{D}$ preserves identity maps. To see that $\sf{D}$ preserves composition, let $(\rho_1, \sigma_1)$ and $(\rho_2, \sigma_2)$ be composable morphisms in $\ldeve$. Their composition is $(\rho_1\star\rho_2, \sigma_1 \circ \sigma_2)$. We have shown above that $\rho_1\star\rho_2$ is then a morphism of $\Dim$. Therefore, by (D5),
\[
(\rho_1\star\rho_2)(b) = \bigvee \{ (\rho_1\star\rho_2)(a) \mid a \in I_{\alpha}, a \prec b\}.
\]
On the other hand, by the definition of the composition $\rho_1\diamond \rho_2$ in $\Dim$,
\[
(\rho_1\diamond \rho_2)(b) = \bigvee \{ \rho_1(\rho_2(a)) \mid  a \in I_\alpha, a \lhd b\}.
\]
By Lemma~\ref{lem: lhd facts}, for $a \in I_\alpha$, we have $a \prec b$ iff $a \lhd b$. Therefore, since
$(\rho_1 \star \rho_2)(a) \le \rho_1(\rho_2(a))$ we get $(\rho_1\star\rho_2)(b) \le (\rho_1 \diamond \rho_2)(b)$. However, by definition of $\star$,
\[
(\rho_1\star\rho_2)(b) = \bigvee \{ (\rho_1(\rho_2(a)) \mid a \prec b\},
\]
which gives the reverse inequality. Therefore, $\rho_1 \star \rho_2 = \rho_1 \diamond \rho_2$, which shows that $\sf{D}$ preserves composition. Thus, $\sf{D}$ is a covariant  functor.
\end{proof}

Our goal is to see that $\sf D$ is an equivalence. For this we need to produce, for a Dimov algebra $\mathfrak{D}$, a maximal locally compact de Vries extension. Let $\mathfrak{D}=(A,\lhd,I)$ be a Dimov algebra. The construction in the following definition is well known in pointfree topology (see, e.g., \cite[p.~126]{Joh82} or \cite[p.~90]{PP12}).

\begin{definition} \label{def: sc prox}
We define $\prec$ on $\mathfrak{D}$ by $a \prec b$ iff there is a family $\{c_p \mid p \in \mathbb{Q} \cap [0,1] \}$ with $a \le c_0$, $c_1 \le b$, and $c_p \lhd c_q$ for each $p < q$. We call the sequence $\{c_p\}$ an \emph{interpolating sequence witnessing} $a \prec b$.
\end{definition}

Recall (see, e.g., \cite[p.~90]{PP12}) that a binary relation $R$ is said to be \emph{interpolating} if $aRb$ implies there is $c$ with $aRc$ and $cRb$. It is standard to see that $\prec$ is the largest interpolating relation contained in $\lhd$.

\begin{remark} \label{rem: repeated}
Suppose that $\mathfrak{D}$ is a Dimov algebra and $\prec$ is given as in Definition~\ref{def: sc prox}. If $a \lhd b$ and $a \in I$, then repeated use of (I1) shows that $a \prec b$.
\end{remark}

In order to prove Theorem~\ref{lem: SW proximity}, we require the following characterization of compactifications of a completely regular space \cite[Thm.~2.2.4]{deV62}, which is de Vries' pointfree version of Smirnov's theorem. If $X$ is a completely regular space, define $\lhd$ on $\RO(X)$ by $U \lhd V$ if $\cl(U) \subseteq V$. If $\prec$ is a proximity on $\RO(X)$, we say that $\prec$ is \emph{compatible} with the topology if $\prec$ is contained in $\lhd$ and $V = \bigcup \{ U \in \RO(X) \mid \exists W \in \RO(X), U \prec W \subseteq V\}$ for each open set $V$.

\begin{theorem}[de Vries] \label{thm: prox=comp}
Let $X$ be a completely regular space. There is an order isomorphism between the poset of (inequivalent) compactifications of $X$ and the poset of proximities $\prec$ on $\RO(X)$ compatible with the topology.
\end{theorem}

\begin{theorem} \label{lem: SW proximity}
Let $\mathfrak{D}=(A,\lhd,I)$ be a Dimov algebra. The relation $\prec$ defined in Definition~\ref{def: sc prox} is a de Vries proximity and $I$ is a round ideal of the de Vries algebra $\dv{A} := (A, \prec)$. Moreover, if $X$ is the open subspace of $Y_\dv{A}$ corresponding to $I$, then $X$ is locally compact and dense in $Y_\dv{A}$, and the inclusion map $e : X \to Y_\dv{A}$ is isomorphic to the Stone-\v{C}ech compactification of $X$. Furthermore, if $\alpha : \dv{A} \to \wp(X)$ is the locally compact de Vries extension corresponding to $e$, then $I = I_\alpha$.
\end{theorem}

\begin{proof}
We first show that $\prec$ is a de Vries proximity.

(DV1). The constant sequence $\{1\}$ is an interpolating sequence, so $1 \prec 1$.

(DV2). If $a \prec b$ and $\{c_p\}$ is an interpolating sequence, then $c_0 \lhd c_1$, so $c_0 \le c_1$. Thus, $a \le b$.

(DV3). If $a \le b \prec c \le d$ and $\{e_p\}$ is an interpolating sequence witnessing $b \prec c$, then it is clear that $\{e_p\}$ is also an interpolating sequence witnessing $a \prec d$.

(DV4). Let $a \prec b,c$ and let $\{d_p\},\{e_p\}$ be interpolating sequences witnessing $a\prec b$ and $a\prec c$, respectively. Set $f_p = d_p \wedge e_q$. Then $\{f_p\}$ is an interpolating sequence witnessing $a \prec (b \wedge c)$.

(DV5). Let $a \prec b$ and $\{c_p\}$ be an interpolating sequence. Set $d_p = \lnot c_{1-p}$. Then $a \le c_0$ and $c_1 \le b$ yield $\lnot b \le d_0$ and $d_1 \le \lnot a$. Moreover, if $p < q$, then $1-q < 1-p$, so $c_{1-q} \lhd c_{1-p}$. Therefore, $d_p = \lnot c_{1-p} \lhd \lnot c_{1-q} = d_q$. Thus, $\{d_q\}$ is an interpolating sequence witnessing $\lnot b \prec \lnot a$.

(DV6). Let $a\prec b$ and let $\{c_p\}$ be an interpolating sequence. Set $c = c_{1/2}$. If $d_p = c_{p/2}$ and $e_p = c_{(1+p)/2}$, then it is well known and straightforward to see that $\{d_p \}$ is an interpolating sequence witnessing $a \prec c$ and $\{e_p\}$ is an interpolating sequence witnessing $c \prec b$.

(DV7). If $b \ne 0$, then there is $0\ne a \in I$ with $a \lhd b$. Thus, $a \prec b$ by a repeated use of (I1).

This proves that $\dv{A}$ is a de Vries algebra. We next show $I$ is a round ideal of $\dv{A}$. Let $a \in I$. Since $a \lhd 1$, by (I1) there is $b \in I$ with $a \lhd b \lhd 1$. Therefore, $a \prec b$. This shows that $I$ is round. In addition, $\bigvee I=1$ since otherwise $\lnot\bigvee I\ne 0$, so by (I3), there is $0\ne a \in I$ with $a \lhd \lnot \bigvee I$. Thus, $a \le \bigvee I,\lnot \bigvee I$, yielding $a = 0$. The obtained contradiction proves that $\bigvee I=1$.

Since $I$ is a round ideal, by Lemma~\ref{prop: closed}, $I$ corresponds to the open subset
$X := \bigcup \{ \zeta(a) \mid a \in I\}$ of $Y_\dv{A}$. As $\bigvee I = 1$, from 
$\zeta(\bigvee I)=\int_{Y_{\dv{A}}}(\cl_{Y_{\dv{A}}}(X))$ it follows that $X$ is dense in $Y_\dv{A}$. 
Since $X$ is an open subset of $Y_\dv{A}$, it is locally compact; and since $X$ is dense, the inclusion map 
$e : X \to Y_\dv{A}$ is a compactification of $X$. Consider the locally compact de Vries extension 
$\alpha : \dv{A} \to \wp(X)$ corresponding to $e:X \to Y_\dv{A}$ and given by $\alpha(a) = \zeta(a) \cap X$. 
We show that $I = I_\alpha$. Let $a\in I$. Since $X=\bigcup \{ \zeta(a) \mid a \in I\}$, we have $\zeta(a)\subseteq X$. 
Because $I$ is round, this implies that $\cl_{Y_{\dv{A}}}(\zeta(a))\subseteq X$. But 
$\cl_{Y_{\dv{A}}}(\zeta(a))=Y_\dv{A} \setminus \zeta(\lnot a)$. So $Y_\dv{A} \setminus \zeta(\lnot a)\subseteq X$, 
Thus, $\lnot\alpha(\lnot a) = X \setminus \alpha(\lnot a) = X \setminus \zeta(\lnot a)$. Because 
$Y_\dv{A} \setminus \zeta(\lnot a) \subseteq X$, we see that 
$\lnot\alpha(\lnot a) = Y_\dv{A} \setminus \zeta(\lnot a) = \cl_{Y_{\dv{A}}}(\zeta(a))$. 
Since $\cl_{Y_{\dv{A}}}(\zeta(a))$ is a compact subset of $X$, we conclude that $a \in I_\alpha$.

Conversely, let $a \in I_\alpha$. Then $X \setminus \alpha(\lnot a)$ is a compact subset of $X$. Since
$X=\bigcup \{ \zeta(b) \mid b \in I\}$, we have $\alpha(a)\subseteq\lnot\alpha(\lnot a)=X \setminus \alpha(\lnot a)\subseteq\bigcup \{ \zeta(b) \mid b \in I\}$. As $X \setminus \alpha(\lnot a)$ is compact and $I$ is an ideal, there is $b\in I$ with $\alpha(a)\subseteq\zeta(b)$. Because $X$ is dense in $Y_\dv{A}$, we have $\cl_{Y_{\dv{A}}}(\zeta(a))=\cl_{Y_{\dv{A}}}(\zeta(a)\cap X)=\cl_{Y_{\dv{A}}}(\alpha(a))\subseteq\cl_{Y_{\dv{A}}}(\zeta(b))$. Since $I$ is a round ideal, there is $c\in I$ with $\cl_{Y_{\dv{A}}}(\zeta(b))\subseteq\zeta(c)$. Therefore, $\zeta(a)\subseteq\zeta(c)$, so $a \le c$, and hence $a\in I$. Thus, $I=I_\alpha$.

We finish the proof by showing that $e$ is the Stone-\v{C}ech compactification of $X$. By a dual argument to that in Remark~\ref{rem: Dimov alternative}, we have $a \lhd b$ iff $a \wedge c \lhd b$ for all $c \in I$. If $a \in I$, then $a \lhd b$ implies $a \prec b$ by Remark~\ref{rem: repeated}. In the following, we will identify $\dv{A}$ with $\RO(Y_\dv{A})$ and write $Y = Y_\dv{A}$. We have
\begin{align*}
U \lhd V &\longleftrightarrow U \cap W \lhd V \quad\textrm{for all } W \in I \\
&\longleftrightarrow U \cap W \prec V \quad\textrm{for all } W \in I \\
&\longleftrightarrow \cl_Y(U \cap W) \subseteq V \quad\textrm{for all } W \in I \\
&\longleftrightarrow \bigcup \{ \cl_Y(U \cap W) \mid W \in I\} \subseteq V.
\end{align*}
We claim that $\bigcup \{ \cl_Y(U \cap W) \mid W \in I\} = \cl_X(U\cap X)$. To see this we first recall that since $I$ is round and $\bigcup I = X$, if $W \in I$, there is $W' \in I$ with $W \prec W'$, and so $\cl_Y(W) \subseteq W' \subseteq X$. Therefore, $\cl_Y(W)$ is a closed subset of $X$, and so $\cl_Y(W) = \cl_X(W)$. From this observation the inclusion $\subseteq$ is clear. For the reverse inclusion, let $x \in \cl_X(U \cap X)$. By definition of $X$, there is $W \in I$ with $x \in W$. If $V$ is an open neighborhood of $x$, then as $x \in \cl_X(U \cap X)$ and $V \cap W$ is an open neighborhood of $x$, we have $U \cap X \cap (V \cap W) \ne \varnothing$, so $(U \cap W) \cap V \ne \varnothing$. Thus, $x \in \cl_Y(U\cap W)$. This verifies the claim. We have therefore shown that $U \lhd V$ iff $\cl_X(U \cap X) \subseteq V \cap X$.

By Lemma~\ref{lem: 4.1}, the map $U \mapsto U \cap X$ is a Boolean isomorphism from $\RO(Y)$ to $\RO(X)$. This allows us to move $\lhd$ to the relation on $\RO(X)$, which we also denote by $\lhd$, given by $U \lhd V$ iff $\cl_X(U) \subseteq V$. Similarly, we can move $\prec$ to a proximity on $\RO(X)$. By a standard Urysohn argument, $\prec$ is the largest proximity on $\RO(X)$ contained in $\lhd$. By Theorem~\ref{thm: prox=comp}, the proximity $\prec$ corresponds to the largest compactification of $X$, and so $e : X \to Y_\dv{A}$ is the Stone-\v{C}ech compactification of $X$. This completes the proof.
\end{proof}

\begin{theorem} \label{thm: Dimov}
The categories $\ldeve$ and $\Dim$ are equivalent.
\end{theorem}

\begin{proof}
By Proposition~\ref{lem: ldeve to Dim}, we have a covariant functor $\sf{D}:\ldeve\to\Dim$. By \cite[Thm.~IV.4.1]{Mac71}, it is sufficient to show that each object of $\Dim$ is isomorphic to the $\sf{D}$-image of an object of $\ldeve$, and that $\sf{D}$ is full and faithful.

Let $\mathfrak{D} = (A, \lhd, I) \in \Dim$. By Theorem~\ref{lem: SW proximity}, we have a locally compact extension $\alpha : \dv{A} \to \wp(X)$ with $I = I_\alpha$. The functor $\sf{D}$ sends this extension to $(A, \lhd', I)$, where $\lhd'$ is defined in the statement of Lemma~\ref{lem: lhd facts}. We show that ${\lhd'} = {\lhd}$. If $a \lhd b$, then $a\wedge c \le a \lhd b \le b \vee \lnot c$, so $(a \wedge c) \lhd (b \vee \lnot c)$ for each $c \in I$. Since $a \wedge c \in I$, we have $(a \wedge c) \prec (b \vee \lnot c)$ by Remark~\ref{rem: repeated}. Thus, $a \lhd' b$. Conversely, suppose that $a \lhd' b$. Then $(a \wedge c) \prec (b \vee \lnot c)$ for all $c \in I$. Therefore, $(a \wedge c) \lhd (b \vee \lnot c)$ by the definition of $\prec$. Thus, $a \lhd b$ by (I2). This proves that $\mathfrak{D} = \sf{D}(\alpha)$.

To see that $\sf{D}$ is faithful, since $\sf{D}(\rho, \sigma) = \rho$ for each morphism $(\rho, \sigma)$ of $\ldeve$, it suffices to show that if $(\rho, \sigma)$ and $(\rho, \psi)$ are morphisms between $\alpha : \dv{A} \to \dv{B}$ and $\alpha' : \dv{A'} \to \dv{B'}$, then $\sigma = \psi$.
\[
\xymatrix@C5pc{
\dv{A} \ar[d]_{\rho} \ar[r]^{\alpha} & \dv{B} \ar@/^/[d]^{\sigma} \ar@/_/[d]_{\psi} \\
\dv{A'} \ar[r]_{\alpha'} & \dv{B'}
}
\]
We have $\sigma \circ \alpha = \psi \circ \alpha$ since both are equal to $\alpha' \star \rho$. Both $\sigma$ and $\psi$ are complete Boolean homomorphisms. Since $\alpha[\dv{A}]$ is join-meet dense in $\dv{B}$, it follows that $\sigma = \psi$. This shows that $\sf{D}$ is faithful.

Finally, to see that $\sf{D}$ is full, for each morphism $\rho : \mathfrak{D} \to \mathfrak{D'}$ in $\Dim$, we need to produce a morphism $(\rho, \sigma)$ in $\ldeve$. Using the construction of Theorem~\ref{lem: SW proximity}, we have locally compact de Vries extensions $\alpha : \dv{A} \to \wp(X)$ and $\alpha' : \dv{A}' \to \wp(X')$ with $I = I_\alpha$ and $I' = I_{\alpha'}$.

We first show that $\rho : \dv{A} \to \dv{A'}$ is a de Vries morphism. Clearly (M1) and (M2) hold. To prove (M3), suppose that $a \prec b$. Then there is an interpolating sequence $\{c_p\}$ witnessing $a \prec b$. For each $p<q$, since $c_p \lhd c_q$, we have $\lnot\rho(\lnot c_p) \lhd \rho(c_q)$ by (D3). Set $d_0 = \lnot\rho(\lnot c_0)$, and $d_p = \rho(c_p)$ if $p > 0$. Then $p < q$ implies $d_p \lhd d_q$ as $d_p \le \lnot\rho(\lnot c_p)$. Moreover, $\lnot\rho(\lnot a) \le \lnot\rho(\lnot c_0) = d_0$. Consequently, $\lnot\rho(\lnot a) \prec \rho(b)$. Finally, for (M4), let $b \in \dv{A}$. Then $\rho(b) = \bigvee \{ \rho(a) \mid a \in I, a \lhd b\}$. However, if $a \in I$ and $a \lhd b$, then $a \prec b$ by Remark~\ref{rem: repeated}. Therefore, $\rho(b) = \bigvee \{ \rho(a) \mid a \prec b\}$, and so (M4) holds. Thus, $\rho$ is a de Vries morphism.

We next show that $\rho_*(X') \subseteq X$. If $x \in X'$, then there is $b \in I'$ with $x \in \zeta(b) \subseteq X'$. By (D4), there is $a \in I$ with $b \le \rho(a)$. Since $I$ is round, there is $c \in I$ with $a \prec c$. We have $b \in x$, so $\rho(a) \in x$. Therefore, $a \in \rho^{-1}(x)$, and so $c \in {\thu \rho^{-1}(x)} = \rho_*(x)$. Thus, $\rho_*(x) \in \zeta(c) \subseteq X$, as desired. The restriction of $\rho_*$ to $X'$ is then a well defined function $X' \to X$, and so there is a complete Boolean homomorphism $\sigma : \wp(X) \to \wp(X')$ given by $\sigma(S) = (\rho_*)^{-1}(S)$ for each $S \subseteq X$.
\[
\xymatrix@C5pc{
\dv{A} \ar[r]^{\alpha} \ar[d]_{\rho} & \wp(X) \ar[d]^{\sigma} \\
\dv{A'} \ar[r]_{\alpha'} & \wp(X')
}
\]
To see that $(\rho, \sigma)$ is a morphism in $\ldeve$, we must show that $\sigma \circ \alpha = \alpha' \star \rho$. Let $b \in \dv{A}$. Then
\begin{align*}
\sigma(\alpha(b)) &= \sigma(\zeta(b) \cap X) = (\rho_*)^{-1}(\zeta(b) \cap X) = \{ x \in X' \mid b \in \rho_*(x) \} \\
&= \{ x \in X' \mid \exists a \prec b : \rho(a) \in x\}.
\end{align*}
On the other hand, $(\alpha' \star \rho)(b) = \bigvee \{ \alpha'(\rho(a)) \mid a \prec b\}$. Now, for $a \prec b$, we have
\[
\alpha'(\rho(a)) = \zeta(\rho(a)) \cap X' = \{ x \in X' \mid \rho(a) \in x \}.
\]
Therefore, as the join in $\wp(X')$ is union,
\[
(\alpha'\star\rho)(b) = \{ x \in X' \mid \exists a\prec b : \rho(a) \in x\} = \sigma(\alpha(b)).
\]
This shows that $(\rho, \sigma)$ is a morphism in $\ldeve$. Since ${\sf D}(\rho, \sigma) = \rho$, we conclude that $\sf{D}$ is full. This completes the proof that $\sf{D}$ is part of a category equivalence between $\ldeve$ and $\Dim$.
\end{proof}

\def\cprime{$'$}
\providecommand{\bysame}{\leavevmode\hbox to3em{\hrulefill}\thinspace}
\providecommand{\MR}{\relax\ifhmode\unskip\space\fi MR }
\providecommand{\MRhref}[2]{%
  \href{http://www.ams.org/mathscinet-getitem?mr=#1}{#2}
}
\providecommand{\href}[2]{#2}

\bigskip

\noindent Department of Mathematical Sciences, New Mexico State University, Las Cruces NM 88003

\bigskip

\noindent guram@nmsu.edu, pmorandi@nmsu.edu, olberdin@nmsu.edu


\begin{thebibliography}{10}

\bibitem{BMO18a}
G.~Bezhanishvili, P.~J. Morandi, and B.~Olberding, \emph{De {V}ries duality for
  compactifications and completely regular spaces}, Submitted. Preprint
  available at arXiv:1804.03210, 2018.

\bibitem{deV62}
H.~de~Vries, \emph{Compact spaces and compactifications. {A}n algebraic
  approach}, Ph.D. thesis, University of Amsterdam, 1962.

\bibitem{Dim10}
G.~D. Dimov, \emph{A de {V}ries-type duality theorem for the category of
  locally compact spaces and continuous maps. {I}}, Acta Math. Hungar.
  \textbf{129} (2010), no.~4, 314--349.

\bibitem{Eng89}
R.~Engelking, \emph{General topology}, second ed., Sigma Series in Pure
  Mathematics, vol.~6, Heldermann Verlag, Berlin, 1989.

\bibitem{GJ60}
L.~Gillman and M.~Jerison, \emph{Rings of continuous functions}, The University
  Series in Higher Mathematics, D. Van Nostrand Co., Inc., Princeton,
  N.J.-Toronto-London-New York, 1960.

\bibitem{Joh82}
P.~T. Johnstone, \emph{Stone spaces}, Cambridge Studies in Advanced
  Mathematics, vol.~3, Cambridge University Press, Cambridge, 1982.

\bibitem{Lea67}
S.~Leader, \emph{Local proximity spaces}, Math. Ann. \textbf{169} (1967),
  275--281.

\bibitem{Mac71}
S.~MacLane, \emph{Categories for the working mathematician}, Springer-Verlag,
  New York, 1971, Graduate Texts in Mathematics, Vol. 5.

\bibitem{NW70}
S.~A. Naimpally and B.~D. Warrack, \emph{Proximity spaces}, Cambridge Tracts in
  Mathematics and Mathematical Physics, No. 59, Cambridge University Press,
  London-New York, 1970.

\bibitem{PP12}
J.~Picado and A.~Pultr, \emph{Frames and locales: Topology without points},
  Frontiers in Mathematics, Birkh\"auser/Springer Basel AG, Basel, 2012.

\end{thebibliography}
\end{document}